\documentclass[12pt,a4paper, reqno]{article}

\usepackage{indentfirst} % identation on the first paragraph
\usepackage{amsmath, amsthm, amssymb,xfrac} % math packages
\usepackage{tikz} % figures
\usepackage{wrapfig} % allows to put text aroung figures
\usepackage{xcolor} % colors
\usepackage{verbatim} % multiline comments
\usepackage{fullpage} % margins
\usepackage{enumerate} % making lists

\usepackage{chngcntr}
\counterwithin{figure}{section}
\counterwithin{equation}{section}

\usepackage{hyperref} % insert this allways last
\hypersetup{
colorlinks=true,
linkcolor=black!50!red,
linktoc=all
}
\usepackage[all]{hypcap} % to correct positioning for the hyperref

\newcounter{constant}
\newcommand{\newconstant}[1]{\refstepcounter{constant}\label{#1}}
\newcommand{\useconstant}[1]{c_{\textnormal{\tiny \ref{#1}}}}
\newcounter{bigconstant}
\newcommand{\newbigconstant}[1]{\refstepcounter{bigconstant}\label{#1}}
\newcommand{\usebigconstant}[1]{C_{\textnormal{\tiny \ref{#1}}}}
\newtheorem{teo}{Theorem}[section]
\newtheorem{prop}[teo]{Proposition}
\newtheorem{lemma}[teo]{Lemma}
\newtheorem{claim}[teo]{Claim}
\newtheorem{cor}[teo]{Corollary}

\theoremstyle{definition}

\newtheorem{ex}[teo]{Example}
\newtheorem{remark}[teo]{Remark}
\newtheorem{question}[teo]{Question}

\newcommand{\PP}{\mathbb{P}}
\newcommand{\EE}{\mathbb{E}}
\newcommand{\RR}{\mathbb{R}}

\newcommand{\NN}{\mathbb{N}}
\newcommand{\ZZ}{\mathbb{Z}}

\newcommand{\charf}[1]{\mathbf{1}_{#1}}

\DeclareMathOperator{\dist}{d}
\DeclareMathOperator{\per}{per}
\DeclareMathOperator{\poisson}{Po}
\DeclareMathOperator{\cov}{Cov}
\DeclareMathOperator{\capacity}{cap}

\begin{document}

\title{How can a clairvoyant particle escape \\ the exclusion process?}

\author{Rangel Baldasso\footnote{Email: \ baldasso@impa.br; \ IMPA, Estrada Dona Castorina 110, 22460-320,\newline Rio de Janeiro, RJ, Brazil} \and Augusto Teixeira\footnote{Email: \ augusto@impa.br; \ IMPA, Estrada Dona Castorina 110, 22460-320,\newline Rio de Janeiro, RJ, Brazil}}

\maketitle

\begin{abstract}
We study a detection problem in the following setting: On the one-dimensional integer lattice, at time zero, place nodes on each site independently with probability $\rho \in [0,1)$ and let them evolve as a simple symmetric exclusion process. At time zero, place a target at the origin. The target moves only at integer times, and can move to any site that is within distance $R$ from its current position. Assume also that the target can predict the future movement of all nodes. We prove that, for $R$ large enough (depending on the value of $\rho$) it is possible for the target to avoid detection forever with positive probability. The proof of this result uses two ingredients of independent interest. First we establish a renormalisation scheme that can be used to prove percolation for dependent oriented models under a certain decoupling condition. This result is general and does not rely on the specifities of the model. As an application, we prove our main theorem for different dynamics, such as independent random walks and independent renewal chains. We also proof existence of oriented percolation for random interlacements and for its vacant set for large dimensions . The second step of the proof is a space-time decoupling for the exclusion process.
\end{abstract}

\begin{comment}
\subjclass[2010]{60K37, 60K35, 82B43, 82C22}
\keywords{Target detection, oriented percolation, exclusion process decoupling.}
\end{comment}

\section{Introduction}
~

% What is it?
\par Suppose we are given a random set of moving points, called nodes. We think of these nodes as detectors. Suppose we also have a target that can be mobile or not. We are interested in knowing whether any of the nodes will detect the target in finite time, and if so, what are the properties of the detection time. Of course the answer to these problems depend on the specific model in question.

% Brownian Motion model
\par There exists a rich literature concerning this class of problems. The mobile geometric graph model is an example of structure where this has been studied. In this model, the starting positions of the nodes is given by a Poisson point process in the plane with intensity $\lambda>0$ and they evolve as independent Brownian Motions. A node detects everything that is within distance at most one from it. This model has been studied under different aspects. When the target is non-mobile, detection occurs in finite time almost surely. In this case, \cite{kkp} derive bounds on the tail distribution of the detection time (the first time a node detects the target). One can also consider a target that moves independently from the nodes, as in \cite{psss}. They also study the tail of the detection time, and which distributions for the movement of the target allows it to avoid detection for the longest time. They prove that, in dimension two, there are two equally good possibilities for it: Stay put or move as an independent Brownian Motion.

% Smart target
\par In the literature presented above, a central hypothesis is that the target moves independently from the nodes. It is also interesting to treat the case when we drop this assumption. We can suppose that the target is able to predict the future trajectories of all the nodes and moves cleverly trying to avoid detection. In \cite{stauffer}, this possibility is considered for the mobile geometric graph for dimensions $d \geq 2$. In this case, they prove a phase transition on the probability of detection as the value of $\lambda$ changes. For dimension one, detection always occurs in finite time for this model.

\bigskip

% Our setting
\par Here we will consider a model in the one-dimensional lattice $\ZZ$. Fix $\rho \in [0,1]$ and place a node in each integer site independently with probability $\rho$. Each node moves as a continuous time random walk obeying the exclusion rule, that says two nodes cannot occupy the same site at the same time. This means that when a node tries to jump to an occupied site, this jump is suppressed. Our initial distribution for the nodes is stationary to this evolution. The collective behaviour of the nodes is the well known exclusion process.

% Target evolution
\par As for the target, it starts at the origin and, unlike the nodes, moves only at integer times. On the other hand, we allow the target to jump to any site within distance at most $R>0$ from its current position. The target is detected if it stays on top of some node. We assume that the target knows the future movement of all the nodes, and we ask if it can scape detection with positive probability.

% Theorem
\par We will prove that, for fixed $\rho \in [0,1)$, there exists a phase transition in the probability of detection in finite time, as we vary the value of $R$.

\begin{teo}\label{teo:detection}
Suppose $\rho \in [0,1)$. There exists $R_{0}=R_{0}(\rho)$ such that if  $R \geq R_{0}$, then the probability that the target is never detected by some node is positive.
\end{teo}

% Why isn't this trivial?
\begin{remark}
It is not the case that for all values of $R$ the target can escape with positive probability. If we take any $\rho>0$ and $R=1$ it is possible find two nodes at time zero, one at each side of the origin. Using a suitable construction of the exclusion process, we conclude that these nodes (as well as the empty sites) move as random walks. This implies that the two nodes we found will eventually meet and strangle the target, who is discovered.
\end{remark}

\par This result shares some similarities with \cite{ss} and \cite{stauffer}. We will get back to these papers later.

\par We use multiscale renormalisation to prove the existence of oriented percolation in dependent models. This allows us to perform some comparison, in a similar flavor of \cite{ss}. The main advantage is that our renormalisation does not rely on the specificities of the model and can be used in other contexts. In Subsection \ref{subsec:applications}, we use this step of the proof to prove existence of oriented percolation for random interlacements and its vacant set, if the dimension is large enough. Besides, in the same subsection, we prove analogous of Theorem \ref{teo:detection} for other underlying dynamics, namely, independent renewal chains and independent random walks. This illustrates different ways to use the renormalisation step of our proof.

\par In \cite{stauffer}, the author uses a multiscale renormalisation to prove the existence of a phase where detection always occurs. Our theorem goes in the opposite direction. It gives sufficient conditions to the existence of percolation, and we use it to prove that, in our model, detection may not happen. To prove that detection may fail, it is necessary to compare the process with oriented percolation, instead of loonking into non-oriented models. This adds a new complicating factor, since oriented paths are harder to exhibit.

\par Theorem \ref{teo:detection} is similar to the one in \cite{ss}. There, the authors consider a model in $\ZZ^{d}$, $d \geq 2$, where nodes are placed according to a Poisson point process with intensity $\lambda$ and move as independent random walks. The target also moves in continuous time, but with bounded speed. They compare this process with oriented percolation to prove a phase transition in the probability of detection as the value of $\lambda$ changes.

\par One may wonder if the techniques from \cite{ss} can be used to prove survival in our setting. There, the authors use a well-chosen subspace of $\ZZ^{d+1}$ and prove that each node only influences a small area. Hence, they can remove the forbidden sites and disregard the trajectory of the node. However, when $d=1$, this fails because, since each node intersects a fixed subspace infinitely many times, the area of influence of each node extends infinitely and can not be disregarded so easily. 

\par This makes the proof in dimension one more intricate, and requires some different machinery. We expect, however, that the proof presented in \cite{ss} can be adapted to our case for larger dimensions.

\par An additional difficulty comes from the choice of the exclusion process as an underlying dynamic due to its lack of good mixing properties, as pointed out in Remark \ref{remark:bad_correation_decay}. For this reason the usual techniques do not apply in a straightforward way. The existence of dependence among the movement of the particles is also a complicating factor.

\par For that reason, another interesting aspect of the paper is the decoupling presented in Section \ref{sec:decoupling}. Decouplig inequalities have been studied in many different contexts, such as random interlacements, in \cite{pt} and \cite{s}, Voronoi percolation, see \cite{br}, and Boolean percolation with random radii, in \cite{att} and \cite{br}. Here we prove a decoupling for a conservatice partilce system that presents dependencies. These techniques should be useful to solve several problems such as understanding the behavior of the random walk on top of the interacting particle system, see \cite{avena}, \cite{hhsst} and \cite{ff}. It is our hope that similar estimates can be proved for other conservative and interactive particle systems, such as the zero-range process.

\bigskip

% Tools
\par There are two general tools we develop in order to establish Theorem \ref{teo:detection}: A renormalisation scheme for oriented percolation models and a decoupling for the exclusion process. These techniques are general facts that we hope can be used in different contexts. The proof of Theorem \ref{teo:detection} lies in the intersection of these techniques that we now briefly describe.

% Renormalisation
\par The first tool is a general statement about percolation in oriented models. We develop a renormalisation scheme that proves percolation using a fixed set of oriented paths. The main advantage of this technique is it does not require independence. Instead, we only need to take care of the decay of correlations on the environment. We will focus here in site percolation in dimension $1+1$, but the proof techniques can be adapted to more general models.

% Decoupling
\par In order to apply our renormalisation scheme, we need to be able to verify the required correlation decay property. This is the content of the second part of the proof, a decoupling for the exclusion process. Our decoupling deals with correlation of functions that depend on the trajectory of the exclusion process only on compact boxes of the space-time $\ZZ \times \RR$. We use a sprinkling argument to obtain good bounds on the error decay in this case. Other potential applications of this estimate should include the study of random walk on the random environment given by the exclusion process, see \cite{avena} and \cite{ff}. In \cite{hhsst}, the authors develop a similar estimate for the particle system formed by independent random walks and prove a law of large numbers and central limit theorem for the random walk over this particle system. We hope that adaptations of their proofs should hold in this case, with the aid of our decoupling.

% Next steps
\par Let us now give a more detailed explanation on each of the two techniques used here. We begin with the introduction of our oriented percolation model. The second subsection is devoted to a discussion on the decoupling for the exclusion process.

\subsection{Overview of oriented dependent percolation}\label{subsec:int_percolation}
~
% The new setting
\par We present here a particular case of the percolation models we are interested in. Let $\mathcal{I} \subset \ZZ^{2}$ be a random subset of the integer lattice with distribution $\PP$ that is translation invariant. Fix also $S$ as the set of paths $f:\NN_{0} \to \ZZ^{2}$ that satisfy
\begin{equation}\label{eq:example}
f(n+1) - f(n) \in \{(0,1),(1,0)\}, \, \text{ for all } n \in \NN_{0}.
\end{equation}
We look for conditions over $\PP$ that ensure the existence of an infinite open path, i.e., a infinite path of $S$ contained in $\mathcal{I}$.

% Independence
\par When $\PP$ is obtained by independently declaring each vertex open with probability $p$, one can easily prove that percolation occurs for large values of $p$, as proved in \cite{grimmett}. Our objective here is to drop the independence assumption. We will assume instead a good decay on the correlations of the environment.

% Correlation decay
\par Let us explain the type of estimate we need. Let $B_{1}$ and $B_{2}$ be two square boxes in $\ZZ^{2}$ which are distant enough, see \eqref{eq:support_distance}. Consider two non-decreasing functions $f_{1}: \{0,1\}^{B_{1}} \to [0,1]$ and $f_{2}:\{0,1\}^{B_{2}} \to [0,1]$. We will assume that
\begin{equation}\label{eq:correlation_decay_percolation}
\limsup \dist(B_{1},B_{2})^{7} \sup \{\text{Cov}(f_{1},f_{2})\} < \frac{1}{200},
\end{equation}
where the supremum is taken over all functions $f_{1}$ and $f_{2}$ that are non-decreasing with respective supports on $B_{1}$ and $B_{2}$ and the $\limsup$ is taken as $\dist(B_{1},B_{2})$ diverges.

% Theorem
\par Our theorem states that if the conditions above are satisfied and the probability that a site is open is big enough, then
\begin{displaymath}
\begin{array}{cl}
\text{the probability that there exists a path} \\ \text{in $S$ that is open is positive.}
\end{array}
\end{displaymath}
This is precisely stated later as Theorem \ref{teo:percolation}.

% More general
\begin{remark}
We want to apply this to prove Theorem \ref{teo:detection}. However, there are some complications in this. 
First of all, the exclusion process does not satisfy the correlation decay in \eqref{eq:correlation_decay_percolation}, as pointed out in Remark \ref{remark:bad_correation_decay}. Therefore, this condition has to be weakened. The second problem is that our set of paths is not given by functions that satisfy \eqref{eq:example}. Finally, we need to modify the condition that deals with the probability of a vertex to be open. Instead, we will use a finite size criterion that will be better explained in Subsection \ref{subsec:box_notation}. We will devote Subsections \ref{subsec:S} and \ref{subsec:P} to the discussion of the general hypothesis we need on the set $S$ of paths and on the probability $\PP$. 
\end{remark}

% Steps of the proof
\par The main step in the proof of our result is a multiscale renormalisation scheme developed to bound the probability of a sequence of events. These events are defined as the absence of crossings of some well chosen sets. We will prove that, if such probabilities present fast enough decay, then it is possible to construct a concatenation of paths in $S$ to obtain an infinite open path.

\par The correlation decay is used to decouple events of this sequence that are far apart. We will deduce a recursive inequality involving the probability of these events. This will allow us to conclude the good decay of the probabilities and hence the existence of the crossings.

\subsection{Overview of decoupling for the exclusion process}\label{subsec:int_decoupling}
~
% Bad correlations
\par As we pointed out in the last subsection, the correlation decay in \eqref{eq:correlation_decay_percolation} is not true for the exclusion process. However, it is possible to prove a similar type of estimate.

% Sprinkling
\par We will work with a stationary exclusion process. The error bounds will be improved by the use of a sprinkling argument, i.e., slightly increasing the density in the exclusion process. This technique has been used in many different contexts, such as \cite{hhsst}, \cite{psss} and \cite{pt}.

% Decoupling
\par The setting is similar to the one described above Equation \eqref{eq:correlation_decay_percolation}. In analogy with \eqref{eq:support}, we say that a function of the whole trajectory of the exclusion process has support in a box $B \subset \ZZ \times \RR_{+}$ if for every pair of configurations $\eta$ and $\xi$
\begin{equation}\label{eq:support_decoupling}
\eta_{t}(x)=\xi_{t}(x) \text{ for all } (x,t) \in B \text{ implies } f(\eta)=f(\xi).
\end{equation}

\bigskip

\par Denote by $\per(B)$ the perimeter of the box $B$. We are ready to state our decoupling.

%%%%
\newconstant{c:decouplingep}
\newbigconstant{C:decouplingep}
%%%%

\begin{teo}[Exclusion process decoupling]\label{teo:decoupling_ep}
There exist positive constants $\useconstant{c:decouplingep}$ and $\usebigconstant{C:decouplingep}$ such that if $f_{1},f_{2}:\{0,1\}^{\ZZ \times \RR} \to [0,1]$ are two non-decreasing functions with respective supports on space-time boxes $B_{1}$ and $B_{2}$ satisfying
\begin{equation}\label{eq:distance_hypothesis}
\dist=\dist(B_{1},B_{2}) \geq 6(\per(B_{1})+\per(B_{2}))+\usebigconstant{C:decouplingep},
\end{equation} 
then, for any densities $\rho < \rho' \in [0,1] $,
\begin{equation}\label{eq:decoupling_estimate}
\EE_{\rho}(f_{1}f_{2}) \leq \EE_{\rho'}(f_1)\EE_{\rho'}(f_2)+\useconstant{c:decouplingep} \dist^{2}\exp\left\{-\useconstant{c:decouplingep}^{-1}(\rho'-\rho)^{2}\dist^{\sfrac{1}{4}}\right\}.
\end{equation}
\end{teo}

\begin{remark}
We can also take $f_{1}$ and $f_{2}$ to be two non-increasing functions and assume that $\rho' < \rho \in [0,1]$. The proof carries out in the same way in this case.
\end{remark}

\begin{remark}
Observe that \eqref{eq:decoupling_estimate} is not a correlation estimate, since we need to add the sprinkling in order to have this bound on the error function.
\end{remark}

% Idea of the proof
\par For the proof of this decoupling we look at the relative position between the boxes. When the horizontal distance between the boxes is large, the result follows easily from concentration bounds. The second case considered is when the vertical distance between the boxes is large but not necessarily the horizontal. In this case, the sprinkling is used. We couple two exclusion processes with densities $\rho < \rho'$ in a way that for a large time, the process with bigger density dominates the less dense process in an interval with large probability.

\par Let us describe the coupling. We will match each particle of the process that has smaller density, with a particle of the process with larger density $\rho'$, similarly to the coupling contructed in \cite{stauffer}. We do this in a careful way so that each pair of particles is not far apart at time zero. Once we have this matching between the particles, we set the evolution of the processes in a way that when a pair of matched particles meets, they stay together from this time on. This is not good enough, since the time it takes for a typical pair to meet does not decay fast enough. To fix this, the matching is remade at some particular times and all the process starts again in order to match more particles.

\bigskip

\textbf{Notation.} Throughout the text, we write $\NN=\{1,2,3,\dots\}$, $\NN_{0}=\NN \cup \{0\}$ and $\RR_{+}=[0,+ \infty)$. We will always use $\PP$ to denote a probability measure and $\EE$ for the expectation with respect to it. When a probability is indexed by some parameter, the expectations with respect to that probability measure will receive the same parameter, e.g., $\EE_{\rho}$ denotes the expected value with respect to $\PP_{\rho}$, the distribution of an exclusion process with density $\rho$.

\bigskip

\textbf{About constants.} Throughout the text, we use $c$ and $C$ to denote universal positive constants. If this constant depends on some variable, we explicitly write it, e.g., $c(\rho)$ denotes a constant that depends on the value of $\rho$. Constants may change from line to line in estimates. Numbered constants refer to their first appearance in the text or to its appearance in the Appendix.

\bigskip

\textbf{Structure of the paper.} Section \ref{sec:percolation} is devoted to the precise statement and proof of our oriented percolation model. In Section \ref{sec:decoupling}, we prove Theorem \ref{teo:decoupling_ep}. The sections corresponding to oriented percolation (\ref{sec:percolation}) and decoupling of the exclusion process (\ref{sec:decoupling}) can be read independently. Finally, in Section \ref{sec:detection} we conclude the proof of Theorem \ref{teo:detection}. This section uses the results in Section \ref{sec:percolation}, but Section \ref{sec:decoupling} can be skipped, since knowing the statement of Theorem \ref{teo:decoupling_ep} is enough.

\bigskip

\textbf{Acknowledgements.} The authors thank Tertuliano Franco, Milton Jara and Roberto Oliveira for valuable discussions on the initial stages of the work. RB thanks CAPES Proex and FAPERJ grant E-26/202.231/2015 for financial support. AT thanks CNPq grants 306348/2012-8 and 478577/2012-5 and FAPERJ grant 202.231/2015 for financial support.

\section{Oriented dependent percolation}\label{sec:percolation}
~
\par In this section we work on our renormalisation scheme for oriented percolation. We begin by working on the details for the construction of the set of paths $S$. The second subsection is dedicated to the study of the probability measure $\PP$. Subsection \ref{subsec:box_notation} is dedicated to the introduction of some additional notation concerning the scales in our renormalisation. Subsection \ref{subsec:tp} we state precisely our theorem and prove it. Finally, in the last subsection we examine some applications of Theorem \ref{teo:percolation}.

\subsection{The set $S$}\label{subsec:S}
~
\par In this subsection, we discuss the properties we need the set of paths.

\par Fix a convex set $\mathcal{C} \subset \RR \times [0,1]$ with $ 0 \in \partial \mathcal{C}$. We will assume that $S$ is formed by all the functions $f: \NN_{0} \to \ZZ^{2}$ such that
\begin{equation}\label{eq:construction_S}
f(n+1)-f(n) \in \mathcal{C} \setminus \{0\}.
\end{equation}

\par We also need to assure the set $S$ is rich enough to allow us to construct crossings of boxes. Hence, we assume
\begin{itemize}
\item[H1.] $(0,1) \in \mathcal{C}$;
\item[H2.] either $(1,0) $ or $(3,1)$ is in $\mathcal{C}$.
\end{itemize}

\par These hypothesis allow the construction of horizontal crossings in boxes of the form $[0,3L] \times [0,L]$ and vertical crossings in boxes of the form $[0,L] \times [0,3L]$, for $L \in \NN$.

\par One of the main reasons why the set $S$ is constructed in this way is a concatenating property we will make use of. For $f \in S$, define $\tilde{f}:\RR_{+} \to \RR^{2}$ as the linear interpolation of $f$:
\begin{equation}\label{eq:linear_interpotarion}
\tilde{f}(t)=(t-\lfloor t \rfloor)f(\lfloor t \rfloor)+(1+\lfloor t \rfloor-t)f(\lfloor t \rfloor+1).
\end{equation}

\par Suppose we are given $f,g \in S$ and that there exist $s,t \in \RR_{+}$ such that $\tilde{f}(s)=\tilde{g}(t)$. Then the concatenation of $f$ and $g$, given by $h: \NN_{0} \to \ZZ^{2}$ as
\begin{equation}\label{eq:concatenation}
h(n)=\begin{cases}
	f(n) & \text{if } n \leq s, \\
	g(\lfloor t \rfloor-\lfloor s \rfloor +n) & \text{if } n > s,
	\end{cases}
\end{equation}
is also in $S$. This is easily verified by observing that
\begin{equation}\label{eq:concatenation_proof}
g(\lfloor t \rfloor+1)-f(\lfloor s \rfloor) \in \mathcal{C}.
\end{equation}

\par We end this subsection with examples of sets that can be considered as the possible paths in our oriented model.

\begin{ex}\label{ex:oriented_percolation_set}
Notice that the set defined in \eqref{eq:example} clearly satisfy all hypothesis above, if we consider the convex set $\mathcal{C}$ to be the convex hull of the points $(0,0)$, $(0,1)$ and $(1,0)$.
\end{ex}

\begin{ex}\label{ex:detection_set}
The second example is important for the proof of Theorem \ref{teo:detection}. Fix $R \geq 3$ and define $\tilde{S}_{R}$ to be the set of paths obtained by using the set $\mathcal{C}_{R}$ given by the convex hull of $(-R,1)$, $(0,0)$ and $(R,1)$. It is easy to see that these sets also satisfy all the hypothesis above.
\end{ex}

\subsection{The probability measure $\PP$}\label{subsec:P}
~
\par In this subsection we state the necessary hypothesis on the measure $\PP$.

\par It will be useful to think of $\PP$ as a measure on $\{0,1\}^{\ZZ^{2}}$ and write $\eta:\{0,1\}^{\ZZ^{2}} \to \{0,1\}$ for the (random) characteristic function given by the (also random) set $\mathcal{I}$.

\par We require the probability $\PP$ to satisfy two conditions that will be discussed in the following.

\par First we assume that
\begin{equation}
\text{$\PP$ is translation invariant}.	
\end{equation}

\par The second condition deals with the decay of correlations. To state this precisely we need some additional notation.

\par Observe that the set $\{0,1\}^{\ZZ^{2}}$ has a partial order given by
\begin{equation}\label{eq:partial_order}
\eta \preceq \xi \text{ if and only if } \eta(x) \leq \xi(x), \text{ for all } x \in \ZZ^{2}.
\end{equation}
This allows us to say that a function $f: \{0,1\}^{\ZZ^{2}} \to \RR$ is non-increasing if
\begin{equation}\label{eq:non_decreasing}
\eta \preceq \xi \text{ implies } f(\eta) \geq f(\xi).
\end{equation}

\par We also say that $f:\{0,1\}^{\ZZ^{2}} \to \RR$ has support on the box $B=[a,b] \times [c,d] \subset \RR^{2}$ if for every pair of configurations $\eta$ and $\xi$
\begin{equation}\label{eq:support}
\eta|_{B \cap \ZZ^{2}}=\xi|_{B \cap \ZZ^{2}} \text{ implies } f(\eta)=f(\xi).
\end{equation}
We set $\per(B)=2(|b-a|+|d-c|)$.

%%%%
\newbigconstant{C:support_distance_percolation_1}
\newbigconstant{C:support_distance_percolation_2}
%%%%

\par We are now ready to state our second assumption on $\PP$. It says that there exist constants $\usebigconstant{C:support_distance_percolation_1}, \usebigconstant{C:support_distance_percolation_2} \geq 0$ such that for any non-increasing functions $f,g:\{0,1\}^{\ZZ^{2}} \to [0,1]$ with respective supports on boxes $B_{1}$ and $B_{2}$ that satisfy
\begin{equation}\label{eq:support_distance}
\dist(B_{1},B_{2}) \geq \usebigconstant{C:support_distance_percolation_1}(\per(B_{1})+\per(B_{2})) +\usebigconstant{C:support_distance_percolation_2},
\end{equation}
we have
\begin{equation}\label{eq:simple_decoupling}
\EE(f(\eta)g(\eta)) \leq \EE(f(\eta))\EE(g(\eta)) + H(\dist(B_{1},B_{2})),
\end{equation}
where the error term $H: \RR \to [0,+\infty) $ is a non-increasing function satisfying \begin{equation}\label{eq:error_decay}
\limsup_{x \to +\infty}x^{7}H(x) < \frac{1}{200}.
\end{equation}

\begin{remark}
Combining equations \eqref{eq:simple_decoupling} and \eqref{eq:error_decay}, is is easy to see that the condition above is analogous to the one in Equation \eqref{eq:correlation_decay_percolation}.
\end{remark}

\begin{remark}
In Equation \eqref{eq:support_distance} above, we will assume that $\usebigconstant{C:support_distance_percolation_1} \geq 1$. This does not weaken our hypothesis, it just simplifies some computations.
\end{remark}

\subsection{The box notation}\label{subsec:box_notation}
~
\par Before stating precisely our theorem we need some notation. We begin by the scale notation that will be used in our renormalisation scheme.

\par First we define the sequence of scales as
\begin{equation}\label{eq:scales}
l_{0}=10^{100}, \qquad l_{k+1}=\lfloor l_{k}^{\sfrac{1}{2}} \rfloor l_{k} \qquad \text{and} \qquad L_{k}=\left\lfloor \left(\frac{3}{2}+\frac{1}{k}\right)l_{k}\right\rfloor.
\end{equation}
Observe that $\frac{l_{k}^{\sfrac{3}{2}}}{2} \leq l_{k+1} \leq l_{k}^{\sfrac{3}{2}}$ and that $l_{k} \leq L_{k} \leq 2l_{k}$ if $k$ is large enough.

\par This allows us to define the sequence of sets (see Figure \ref{fig:crossing})
\begin{equation}\label{eq:rectangle}
A_{k}=[0,l_{k}]\times[0,L_{k}]\cup [0,l_{k}+L_{k}]\times[L_{k},L_{k}+l_{k}].
\end{equation}
We also set the box of $A_{k}$ as
\begin{equation}\label{eq:box}
B_{k}=[0,l_{k}+L_{k}] \times [0,l_{k}+L_{k}].
\end{equation}

\begin{figure}
\centering
\begin{tikzpicture}[scale=0.5]

% Box
\draw[-, thick]  (0,0)--(3,0)--(3,3.5)--(6.5,3.5)--(6.5,6.5)--(0,6.5)--(0,0);

\draw[thick, <->] (3.1,0)--(3.1,3.5);
\node[right] at (3,1.75) {$L_{k}$};

\draw[thick, <->] (0,-0.1)--(3,-0.1);
\node[below] at (1.5,-0.1) {$l_{k}$};

% Crossing
\fill[black] (2,0) circle (0.1);
\fill[black] (2.5,0.5) circle (0.1);
\fill[black] (1.5,1) circle (0.1);
\fill[black] (1.5,1.5) circle (0.1);
\fill[black] (2.5,2) circle (0.1);
\fill[black] (2,2.5) circle (0.1);
\fill[black] (2.5,3) circle (0.1);
\fill[black] (2.5,3.5) circle (0.1);
\fill[black] (3,4) circle (0.1);
\fill[black] (4,4.5) circle (0.1);
\fill[black] (5,5) circle (0.1);
\fill[black] (4.5,5.5) circle (0.1);
\fill[black] (5.5,6) circle (0.1);
\fill[black] (6.5,6) circle (0.1);

\draw[-]  (2,0)--(2.5,0.5)--(1.5,1)--(1.5,1.5)--(2.5,2)--(2,2.5)--(2.5,3)--(2.5,3.5)--(3,4)--(4,4.5)--(5,5)--(4.5,5.5)--(5.5,6)--(6.5,6);

\end{tikzpicture}
\caption{The set $A_{k}$ and a crossing of it.}\label{fig:crossing}
\end{figure}
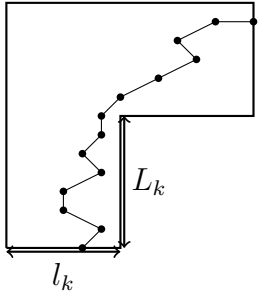

\par Recall the linear interpolation of a function $f \in S $, defined in \eqref{eq:linear_interpotarion}. We say that $f \in S$ is a crossing of $A_{k}$ (see Figure \ref{fig:crossing}) if there exists $T_{f} \in \RR_{+}$ such that

\begin{itemize}
\item[] $\tilde{f}(0)=f(0) \in [0,l_{k}] \times \{0\}$;
\item[] $\tilde{f}(T_{f}) \in \{l_{k}+L_{k}\} \times [L_{k}, l_{k}+L_{k}]$;
\item[] $f(n) \in A_{k}$ , for all $n \in [0,T_{f}] \cap \NN_{0}$.
\end{itemize}

\par We say that a crossing $f$ of $A_{k}$ is open if $\eta(f(n))=1$, for all $n \in [0,T_{f}] \cap \NN_{0}$. Define the events
\begin{equation}\label{eq:renormalization_event}
D_{k}=\left\{\text{there exists no open crossing $f$ of } A_{k} \right\}.
\end{equation}
These are the events whose probability we are interested in bounding. We also define
\begin{equation}\label{eq:renormalization_probability}
p_{k}(S)=\PP[D_{k}].
\end{equation}

\par Although the probabilities $p_{k}(S)$ depend on the set $S$, we will usually omit this dependence and write only $p_{k}$.

\par An important observation is that the event $D_{k}$ has support in the box $B_{k}$, in the sense of \eqref{eq:support}. Notice also that the characteristic function of $D_{k}$ is a non-increasing function.

\par For $x \in \ZZ^{2}$, define the translated sets $A_{k}(x)$, $B_{k}(x)$ and write $D_{k}(x)$ for the event in \eqref{eq:renormalization_event}, when replacing $A_{k}$ by $A_{k}(x)$ in its definition.

\subsection{Theorem and proof}\label{subsec:tp}
~
\par We begin by stating the theorem we prove in this subsection

\begin{teo}\label{teo:percolation}
Suppose $\PP$ satisfy all the hypothesis in Subsection \ref{subsec:P}. There exists a $\tilde{k} \in \NN$ such that, for any set $S$ satisfying the hypothesis in Subsection \ref{subsec:S}, if
\begin{equation}\label{eq:trigger}
p_{k} \leq l_{k}^{-4}, \qquad \text{for some } k \geq \tilde{k},
\end{equation}
then
\begin{equation}\label{eq:recurrence}
p_{n} \leq l_{n}^{-4}, \qquad \text{for all } n \geq k.
\end{equation}
Besides,
\begin{equation}\label{eq:existence_of_percotaion}
\PP[\textnormal{there exists an infinite open path } f \in S]>0.
\end{equation}
\end{teo}

\begin{remark}
The value of $\tilde{k}$ does not depend on the set $S$. In fact, its dependence on the probability measure $\PP$ is only through the error function $H$ in \eqref{eq:simple_decoupling}.
\end{remark}

\par The proof of this theorem begins with a lemma that relates the events $D_{k}$ and $D_{k-1}$. We will prove that if $D_{k}$ holds, then there exists two events in the scale $k-1$ that hold and are far apart, in the sense of \eqref{eq:support_distance}.

\begin{lemma}\label{lemma:renormalization}
There exists $k_{0} \in \NN$ such that for each $ k \geq k_{0}$, there exists $M_{k} \in \ZZ^{2}$ satisfying
\begin{itemize}
\item[1.] $|M_{k}| \leq 10l_{k-1}^{1/2} $;
\item[2.] If $D_{k}$ happens, there exists $x,y \in M_{k}$ such that $D_{k-1}(x)$ and $D_{k-1}(y)$ happen and
\begin{equation}\label{eq:support_distance_renormalization}
\dist(B_{k-1}(x),B_{k-1}(y)) \geq \usebigconstant{C:support_distance_percolation_1}(\per(B_{k-1}(x))+\per(B_{k-1}(y))) +\usebigconstant{C:support_distance_percolation_2},
\end{equation}
with the constants $\usebigconstant{C:support_distance_percolation_1}$ and $\usebigconstant{C:support_distance_percolation_2}$ as in \eqref{eq:support_distance}.
\end{itemize}
\end{lemma}

\begin{proof}
We will look into the event $D_{k}$ for fixed $k$. The idea of the proof is to construct two chains of events in the scale $k-1$ in a way that, if $D_{k}$ holds, then one event in each chain necessarily holds.

We will construct a chain of sets of the form $A_{k-1}$ and take the corresponding events $D_{k-1}$. First, define
\begin{equation}\label{eq:first_chain}
x_{j}=j(l_{k-1},L_{k-1}), \,\,\, 0 \leq j \leq \frac{L_{k}+l_{k}}{L_{k-1}}.
\end{equation}
Observe that $(A_{k-1}(x_{j}))_{0 \leq j \leq \frac{L_{k}+l_{k}}{L_{k-1}}}$ crosses the set $A_{k}$ from the bottom to the top, as in Figure \ref{fig:chain_1}. Notice that the sequence $(A_{k-1}(x_{j}))_{0 \leq j \leq \frac{L_{k}+l_{k}}{L_{k-1}}}$ does not touch the point $(l_{k},L_{k})$. This is a simple consequence of
\begin{equation}\label{eq:touching_corner}
l_{k-1}\frac{L_{k}}{L_{k-1}}+L_{k-1} \leq L_{k}.
\end{equation}

\begin{figure}[h]
\centering
\begin{tikzpicture}[scale=0.6]

% Box
\draw[-, thick]  (0,0)--(3,0)--(3,3.5)--(6.5,3.5)--(6.5,6.5)--(0,6.5)--(0,0);
\draw[-, thick, shift={(7.5,0)}]  (0,0)--(3,0)--(3,3.5)--(6.5,3.5)--(6.5,6.5)--(0,6.5)--(0,0);

% Chain 1
\foreach \x in {0,1, ...,15}{

\draw[-, shift={(0.15*\x,0.4*\x)}] (0,0)--(0.15,0)--(0.15,0.4)--(0.55,0.4)--(0.55,0.55)--(0,0.55)--(0,0);
\draw[-, shift={(13.45-0.4*\x,5.95-0.15*\x)}] (0,0)--(0.15,0)--(0.15,0.4)--(0.55,0.4)--(0.55,0.55)--(0,0.55)--(0,0);

}

\end{tikzpicture}
\caption{The first collection of sets and its reflection.}\label{fig:chain_1}
\end{figure}
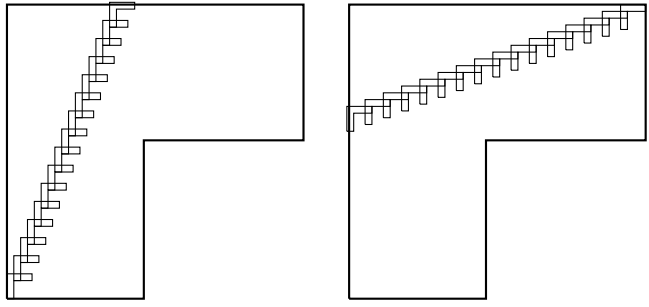

Reflecting this construction across the diagonal of the set $A_{k}$ it is possible to find a sequence that connects the left boundary of $A_{k}$ to its upper right boundary (see Figure \ref{fig:chain_1}).

We now take the first chain of events to be the corresponding events $D_{k-1}$, i.e., take $D_{k-1}(x)$ for the values of $x$ in \eqref{eq:first_chain} or in its reflection. This concludes the construction of the first chain.

For the second chain we consider
\begin{equation}\label{eq:second_chain}
y_{j}=(l_{k},L_{k})-j(l_{k-1},L_{k-1}), \,\,\, 1 \leq j \leq \frac{L_{k}}{L_{k-1}}.
\end{equation}
Again in this case we use a reflection argument and construct the events as in the first chain (see Figure \ref{fig:chain_2}).

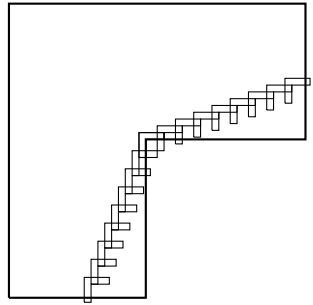
\begin{figure}[h]
\centering
\begin{tikzpicture}[scale=0.6]

% Box
\draw[-, thick]  (0,0)--(3,0)--(3,3.5)--(6.5,3.5)--(6.5,6.5)--(0,6.5)--(0,0);

% Chain 2
\foreach \x in {0,1, ..., 8}{
\draw[-, shift={(2.85-0.15*\x,3.1-0.4*\x)}] (0,0)--(0.15,0)--(0.15,0.4)--(0.55,0.4)--(0.55,0.55)--(0,0.55)--(0,0);
\draw[-, shift={(2.85+0.4*\x,3.1+0.15*\x)}] (0,0)--(0.15,0)--(0.15,0.4)--(0.55,0.4)--(0.55,0.55)--(0,0.55)--(0,0);
}

\end{tikzpicture}
\caption{The second chain constructed.}\label{fig:chain_2}
\end{figure}

Take $M_{k}$ to be the set of all points $x \in \ZZ^{2}$ such that $A_{k-1}(x)$ is in some of the two chains described above. Observe that, by \eqref{eq:first_chain} and \eqref{eq:second_chain},
\begin{equation*}
|M_{k}| \leq 2\left(\frac{L_{k}+l_{k}}{L_{k-1}}+\frac{L_{k}}{L_{k-1}}\right) \leq 10l_{k-1}^{1/2},
\end{equation*}
that is exactly the first conclusion of the lemma.

Now, suppose that $D_{k}$ holds. We will prove that one event in each of the two chains necessarily occurs. Suppose not, and assume, without loss of generality, that all events in the second chain do not happen. In this case, every set $A_{k-1}(x)$ with $x$ as in \eqref{eq:second_chain} or in its reflection, has an open crossing by some function of $S$. If we concatenate these open paths (see Equation \eqref{eq:concatenation}), we obtain an open crossing of $A_{k}$, contradicting out assumption that $D_{k}$ holds.

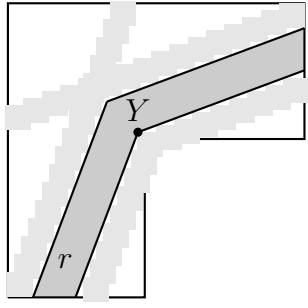
\begin{figure}[h]
\centering
\begin{tikzpicture}[scale=0.6]

% Strip

\fill[black!20!] (0.55,0)--(1185/800,0)--(2.85, 3.65)--(6.5, 6.5-1185/800)--(6.5,5.95)--(239/110, 238/55);

% Box
\draw[-, thick]  (0,0)--(3,0)--(3,3.5)--(6.5,3.5)--(6.5,6.5)--(0,6.5)--(0,0);

% Chain 1
\foreach \x in {0,1, ...,15}{

\fill[black!10!, shift={(0.15*\x,0.4*\x)}] (0,0)--(0.55,0)--(0.55,0.55)--(0,0.55)--(0,0);
\fill[black!10!, shift={(5.95-0.4*\x,5.95-0.15*\x)}] (0,0)--(0.55,0)--(0.55,0.55)--(0,0.55)--(0,0);
}

% Chain 2
\foreach \x in {0,1, ..., 8}{
\fill[black!10!, shift={(2.85-0.15*\x,3.1-0.4*\x)}] (0,0)--(0.55,0)--(0.55,0.55)--(0,0.55)--(0,0);
\fill[black!10!, shift={(2.85+0.4*\x,3.1+0.15*\x)}] (0,0)--(0.55,0)--(0.55,0.55)--(0,0.55)--(0,0);
}

% Lines
\draw[-, thick] (0.55,0)--(239/110, 238/55);
\draw[-, thick] (6.5,5.95)--(239/110, 238/55);

\draw[-, thick] (1185/800,0)--(2.85, 3.65);
\draw[-, thick] (6.5, 6.5-1185/800)--(2.85, 3.65);

% Labels

\node[above] (a) at (2.85, 3.65){$Y$};
\fill[black] (2.85, 3.65) circle (0.1);
\node[right] (a) at (0.85, 0.8){$r$};
\end{tikzpicture}
\caption{The strip, the point $Y$ and the line $r$. Notice that we represented the boxes corresponding to the chains instead of the sets $A_{k-1}$ and that the line $r$ is defined by the vertices of these boxes.}\label{fig:strip}
\end{figure}

To get the distance estimate in Equation \eqref{eq:support_distance_renormalization} observe that there is a strip (see Figure \ref{fig:strip}) that splits the two chains. As a consequence, we can bound the distance in \eqref{eq:support_distance_renormalization} by the distance between the line $r$ and the point $Y$ of Figure \ref{fig:strip}. Since the line $r$ has equation
\begin{equation}\label{eq:line_r}
y=\frac{L_{k-1}}{l_{k-1}}\left(x-L_{k-1}\right),
\end{equation}
a simple computations yields
\begin{equation}\label{eq:estimate_1}
\begin{split}
\dist(B_{k-1}(x),B_{k-1}(y)) & \geq \dist(r,Y) \\
& = \frac{\left|L_{k-1}l_{k}-L_{k-1}l_{k-1}-l_{k-1}L_{k}-l_{k}^{2}-L_{k-1}^{2} \right|}{\sqrt{l_{k-1}^{2}+L_{k-1}^{2}}} \\
& \geq l_{k-1}\frac{1}{\sqrt{5}l_{k-1}}\left[L_{k}-\lfloor l_{k-1}^{1/2}\rfloor L_{k-1}+L_{k-1}^{2}+l_{k-1}^{2}+\frac{L_{k-1}^{2}}{l_{k-1}}\right] \\
& \geq l_{k-1} \frac{\lfloor l_{k-1}^{1/2}\rfloor}{\sqrt{5}} \left(\frac{1}{k}-\frac{1}{k-1}\right),
\end{split}
\end{equation}
for $k$ large enough.

Now, since $l_{k}$ has super-exponential growth (see \eqref{eq:scales}), it is easy to conclude that
\begin{equation}\label{eq:estimate_2}
\lim_{k \rightarrow \infty}\lfloor l_{k-1}^{1/2}\rfloor \left(\frac{1}{k}-\frac{1}{k-1}\right)=+\infty
\end{equation}

If we combine equations \eqref{eq:estimate_1} and \eqref{eq:estimate_2} and use that $\per(B_{k-1}) \leq 12 l_{k-1}$, it is easy to conclude that for $k$ large enough we have
\begin{equation*}
\dist(B_{k-1}(x),B_{k-1}(y)) \geq \usebigconstant{C:support_distance_percolation_1}(\per(B_{k-1}(x))+\per(B_{k-1}(y))) +\usebigconstant{C:support_distance_percolation_2},
\end{equation*}
which is exactly Estimate \eqref{eq:support_distance_renormalization}.
\end{proof}

\par The lemma above provides us with a way to estimate the probability $p_{k}$ in terms of $p_{k-1}$, if $k$ is large. Since we know the realisation of $D_{k}$ implies that two events of order $k-1$ with indices in $M_{k}$ hold, and that they satisfy \eqref{eq:support_distance_renormalization}, we can use \eqref{eq:simple_decoupling} with an union bound to obtain
\begin{equation}\label{eq:probability_decay}
p_{k+1} \leq |M_{k}|^{2}(p_{k}^{2}+H(\usebigconstant{C:support_distance_percolation_1} l_{k})),
\end{equation}
since the distance between the boxes is at least $\usebigconstant{C:support_distance_percolation_1} l_{k}$ and the error function $H$ is non-increasing. This will help us to conclude the proof of Theorem \ref{teo:percolation}, our next goal.

\begin{proof}[Proof of Theorem \ref{teo:percolation}.]

Take $\tilde{k} \geq k_{0}$ so that, for all $k \geq \tilde{k}$
\begin{equation}\label{eq:tilde_k_choice}
100\left(l_{k}^{-1}+l_{k}^{7}H(\usebigconstant{C:support_distance_percolation_1}l_{k})\right) \leq 1,
\end{equation}
where $H$ is the error function in \eqref{eq:simple_decoupling} and $\usebigconstant{C:support_distance_percolation_1}$ is the constant in \eqref{eq:support_distance}. Observe that this is possible by \eqref{eq:error_decay}, since $l_{k} \rightarrow +\infty$ as $k \rightarrow \infty$.

Suppose now that \eqref{eq:trigger} holds, i.e., for some $ k \geq \tilde{k}$, $p_{k} \leq l_{k}^{-4}$. Inductively, using \eqref{eq:probability_decay} we get
\begin{align*}
l_{n+1}^{4}p_{n+1} & \leq 100l_{n+1}^{4}l_{n}\left(p_{n}^{2}+H(\usebigconstant{C:support_distance_percolation_1}l_{n})\right)  \\
& \leq 100l_{n}^{7}\left(l_{n}^{-8}+H(\usebigconstant{C:support_distance_percolation_1}l_{n})\right) \leq 1,
\end{align*}
which concludes the proof of \eqref{eq:recurrence}.

Let us now verify that percolation occurs with positive probability. We will use an adaptation of the construction in the proof of Lemma \ref{lemma:renormalization}. Begin by observing that
\begin{equation}\label{eq:borel_cantelli}
\sum_{k=1}^{\infty}10l_{k}^{1/2}p_{k} < \infty.
\end{equation}

For each $k$, let $U_{k} \subset M_{k}$ to be the set of points $x \in \ZZ^{2}$ such that $D_{k-1}(x)$ is in the second chain constructed in the proof of Lemma \ref{lemma:renormalization}. By Borel-Cantelli Lemma, \eqref{eq:borel_cantelli} implies that only finitely many events in the collection $\{D_{k-1}(x), x \in U_{k} , k \in \NN\}$ can hold. Thus, we may assume that $D_{k-1}(x)$ does not hold for all $x \in U_{k}$ and all sufficiently large $k$. This implies that for each of these points $x$ it is possible to find an open crossing of $A_{k-1}(x)$ by some function of $S$. We use a concatenation of the crossings to find an infinite open path $f \in S$, which concludes the proof.
\end{proof}

\begin{remark}\label{remark:vacant_percolation}
If one is interested in the vacant set, the verification of Equation \eqref{eq:simple_decoupling} for non-decreasing functions allows to prove an analogous result from Theorem \ref{teo:percolation}, but looking for closed paths in $S$.
\end{remark}

\par It may be the case that the probability measure $\PP$ allows us to construct a family $(\mathcal{I}_{u})_{u \in \mathcal{U}}$, with either $\mathcal{U}=[0,1]$ or $\mathcal{U}=\RR_{+}$, of increasing subsets of $\ZZ^{2}$. In this case we can replace the correlation decay \eqref{eq:simple_decoupling} by the inequality
\begin{equation}\label{eq:sprinkling_decoupling}
\EE_{u}(fg) \leq \EE_{u(1-\varepsilon)}(f)\EE_{u(1-\varepsilon)}(g)+H(\varepsilon, d(B_{1},B_{2})),
\end{equation}
with error function $H:\RR_{+}\times \RR_{+} \to \RR_{+}$ that is non-increasing in each of the variables and satisfies that, for some $\delta>0$,
\begin{equation}\label{eq:error_decay_sprinkling}
\lim_{x \to +\infty}x^{7}H(x^{-\delta},x)=0.
\end{equation}

\par Theorem \ref{teo:percolation} can be extended for such values of $u \in \mathcal{U}$. Fix $u_{\infty} \in \mathcal{U}$, set
\begin{equation}\label{eq:u_samll}
u_{0}=u_{\infty}\prod_{k=0}^{\infty}(1-l_{k}^{-\delta}), \,\,\,\,\, 
 u_{k+1}=\frac{u_{k}}{(1-l_{k}^{-\delta})},
\end{equation}
and define
\begin{equation}\label{eq:renormalization_probability_sprinkling}
p_{k}=\PP_{u_{k}}[D_{k}].
\end{equation}

\par With these definitions, the proof of Theorem \ref{teo:percolation} carries on in the same way. In \eqref{eq:borel_cantelli}, we can replace $p_{k}$ by $\PP_{u_{\infty}}(D_{k})$  just by noticing that $\PP_{u}(D_{k})$ is non-increasing in $u \in \mathcal{U}$.

\par This generalization states that in order to conclude the existence of percolation for the density $u_{\infty}$, one needs to understand the probability of the existence of crossings in smaller densities.

\begin{remark}
The error decay on \eqref{eq:error_decay_sprinkling} is not sharp and may be modified to fit in other cases. Sometimes, for example, it may be the case that the error depends also on $u$ and not on $\varepsilon$ as in \eqref{eq:sprinkling_decoupling}.  In these cases, we only need to find $\tilde{k}$ large enough so that \eqref{eq:tilde_k_choice} holds for all $k \geq \tilde{k}$. To do so, one can change the scales $u_{k}$, but not the ones in \eqref{eq:scales}, since the proof of Lemma \ref{lemma:renormalization} strongly uses their growth rate.
\end{remark}

\subsection{Applications}\label{subsec:applications}
~
\par In this subsection we work on examples where Theorem \ref{teo:percolation} can be applied.

\bigskip
\bigskip

\par \textbf{Oriented percolation.} An easy application of Theorem \ref{teo:percolation} is to prove the existence of a supercritical phase in i.i.d. oriented site percolation. Here, we set $S$ as in Example \ref{ex:oriented_percolation_set} and look for a path in $S$ that only crosses open vertices. Independence trivially implies our decoupling assumption. This implies the existence of $n_{0}$ such that, if the probability of not crossing $A_{n_{0}}$ is small enough, then percolation using paths in $S$ is possible with positive probability. This is true if $p$ is large enough, since this probability tends to zero as $p$ goes to one.

\bigskip
\bigskip

\par \textbf{Oriented percolation for random interlacements.} The model of random interlacements was introduced by A.S. Sznitman in \cite{s}. This model is defined in $\ZZ^{d}$, for $d \geq 3$, and consists of a Poissonian cloud of doubly-infinite random walk trajectories. A non-negative parameter $u$ measures the intensity of the model, and $\mathcal{I}^{u}$ denotes the intelacement set at level $u$. Percolation in this setting has been extensively studied. Our focus is in proving oriented percolation in this setting using the set of paths $S$ given by Example \ref{ex:oriented_percolation_set}.

\par Although this model is defined for dimensions at least three, we study the behaviour of its intersection with a plane. Observe that this is not a limitation because the existence of an oriented path in the plane implies the existence of oriented percolation. We prove that if the parameter $u$ is large enough, then oriented percolation is possible in this plane.

\begin{teo}[Oriented percolation for random interlacements]
If the parameter $u$ is large enough (depending on the dimension $d$), then
\begin{equation}
\PP[\mathcal{I}^{u} \cap \ZZ^{2} \text{ contains an infinite path of } S]=1.
\end{equation}
\end{teo}

\par We use the extended version of Theorem \ref{teo:percolation}. The verification of \eqref{eq:sprinkling_decoupling} is done in \cite{pt}, and we restate it here:

%%%%
\newconstant{c:decoupling_random_interlacements}
%%%%

\begin{teo}[Theorem 1.1 of \cite{pt}] \label{teo:pt}
Let $A_{1}$ and $A_{2}$ be two non-intersecting subsets of $\ZZ^{d}$, with at least one of them finite. Let $s$ be the distance between them and $r$ be the minimum of their diameters. Then, for all $u > 0$ and $\varepsilon \in (0,1)$, 

\begin{itemize}

\item[i.] for any non-decreasing functions $f_{1}:\{0,1\}^{A_{1}} \rightarrow [0,1]$ and $f_{2}:\{0,1\}^{A_{2}} \rightarrow [0,1]$
\begin{equation}\label{eq:decoupling_random_interlacements}
\EE_{u}(f_{1}f_{2}) \leq \EE_{u(1+\varepsilon)}(f_{1})\EE_{u(1+\varepsilon)}(f_{2}) + \useconstant{c:decoupling_random_interlacements}(r+s)\exp(-\useconstant{c:decoupling_random_interlacements}^{-1}\varepsilon^{2}u s^{d-2});
\end{equation}

\item[ii.] for any non-increasing functions $f_{1}:\{0,1\}^{A_{1}} \rightarrow [0,1]$ and $f_{2}:\{0,1\}^{A_{2}} \rightarrow [0,1]$
\begin{equation}\label{eq:decoupling_random_interlacements_2}
\EE_{u}(f_{1}f_{2}) \leq \EE_{u(1-\varepsilon)}(f_{1})\EE_{u(1-\varepsilon)}(f_{2}) + \useconstant{c:decoupling_random_interlacements}(r+s)\exp(-\useconstant{c:decoupling_random_interlacements}^{-1}\varepsilon^{2}u s^{d-2}).
\end{equation}
\end{itemize}
\end{teo}

\par If we assume $u \geq 1$, the theorem above verifies \eqref{eq:sprinkling_decoupling} for any $\delta < \sfrac{1}{2}$. This assures us the existence of $k_{0}$ such that, if the probability of not having a crossing of $A_{k_{0}}$ is small enough, then percolation is possible. Now, as $u \rightarrow \infty$, $u_{k_{0}} \rightarrow \infty$ and hence
\begin{equation}\label{eq:trigger_random_nterlacement}
\PP\left[A_{k_{0}} \subset \mathcal{I}^{u_{k_{0}}}\right] \rightarrow 1.
\end{equation}
This implies oriented percolation for large values of $u$.

\bigskip
\bigskip

\par \textbf{The vacant set of random interlacements for large dimensions.} We keep working on the random interlacements model, but now we focus on the vacant set. We want to prove that if $u$ is small enough, then oriented percolation occurs in the vacant set.

\par We will restrict ourselves in $u \in [0,1]$ and define the vacant set as
\begin{equation}\label{eq:vacant_set}
\mathcal{V}_{u}=\ZZ^{d} \setminus \mathcal{I}^{u}.
\end{equation}
Notice that the vacant sets form a non-increasing collection of random subsets. Since we are interested in the vacant set, we use Remark \ref{remark:vacant_percolation}.

%%%%
\newconstant{c:decoupling_random_interlacements_2}
%%%%

\par It follows from (2.15) of \cite{s} that, if $f$ and $g$ are two non-decreasing functions of the vacant set, with respective support on boxes $B \subset \ZZ^{2} \times \{0\}^{d-2}$ and $B+x \subset \ZZ^{2} \times \{0\}^{d-2}$, then
\begin{align}\label{eq:polynomial_decoupling}
\EE(f(\mathcal{V}_{u})g(\mathcal{V}_{u})) & \leq \EE(f(\mathcal{V}_{u}))\EE(g(\mathcal{V}_{u}))+\useconstant{c:decoupling_random_interlacements_2}\frac{u \capacity(B)^{2}}{d(B,B+x)^{d-2}} \\
& \leq \EE(f(\mathcal{V}_{u}))\EE(g(\mathcal{V}_{u}))+\useconstant{c:decoupling_random_interlacements_2}\frac{|B|^{2}}{d(B,B+x)^{d-2}}. \nonumber
\end{align}

\par This implies that if $d > 13$, assuming \eqref{eq:support_distance}, we have \eqref{eq:simple_decoupling} with an error that satisfies \eqref{eq:error_decay}.

\par Since this error is uniform on the parameter $u$, we can find $k_{0}$ such that, if
$\PP_{u}(D_{k_{0}})$ is small enough (here, $D_{k_{0}}$ looks into the existence of open crossings on the vacant sets), then percolation occurs with positive probability on the vacant set. But once again we have
\begin{displaymath}
\lim_{u \rightarrow 0} \PP_{u}[D_{k_{0}}] =0,
\end{displaymath}
and this concludes the proof.

\begin{remark}
The reason why we use the polynomial bound in \eqref{eq:polynomial_decoupling} instead of the exponential bound given by Theorem \ref{teo:pt} is the dependence of the later on the parameter $u$. We need to make $u \rightarrow 0$ and the bounds of Theorem \ref{teo:pt} get worse as $u$ decreases, while \eqref{eq:polynomial_decoupling} is true for all values of $u \in [0,1]$.
\end{remark}

\begin{question}
It remains to prove that oriented percolation occurs for the vacant set in all dimensions $d$. We suspect this is true. For a proof, one should find inspiration on the techniques developed in \cite{ssz}.
\end{question}

\bigskip
\bigskip

\par \textbf{Independent renewal chains.} As an application, we prove an analogous of Theorem \ref{teo:detection} for a different environment. This is the same setting described in Subsection 3.5 from \cite{hhsst}.

\par Fix the probability distribution $p=(p_{n})_{n \in \NN_{0}}$ in $\NN_{0}$ given by $p_{n}=Z^{-1}e^{-n^{\sfrac{1}{4}}}$, where $Z= \sum_{n \in \NN_{0}} e^{-n^{\sfrac{1}{4}}}$ is a normalizing constant. Define the transition rates as
\begin{equation*}
g(l,m)=
\begin{cases}
\delta_{l-1}(m), \,\, & \text{ if } l>0; \\
p_{m}, \,\, & \text{ otherwise}.
\end{cases}
\end{equation*}

\par This Markov chain is called the renewal chain with interarrival distribution $p$. Its stationary measure $q$ is given by
\begin{equation*}
q_{n}=\frac{1}{Z'}\sum_{j \geq n} e^{-j^{\sfrac{1}{4}}} \,\,\,\, \text{ and }\,\,\,\, Z'=\sum_{n \in \NN_{0}}\sum_{j \geq n} e^{-j^{\sfrac{1}{4}}}.
\end{equation*}

%%%%%
\newbigconstant{c:decoupling_irc}
\newconstant{c:decoupling_irc_small}
%%%%%

\par Now, for each $x \in \ZZ$, we consider an independent copy $(N_{x}(n))_{n \in \NN_{0}}$ of the Markov chain described above.  We prove that survival in the empty sites is possible for the set of paths $S_{R}$ from Example \ref{ex:detection_set}, if $R$ is large enough. Here we consider a site $(x,t) \in \ZZ^{2}$ open if $N_{x}(t)=0$.

\par Denote by $\PP_{q}$ the probability measure induced by this process with initial distribution given by independent copies of $q$ in each coordinate of $\ZZ$. In \cite{hhsst}, the authors prove that, for non-increasing functions $f_{1},f_{2}: \NN_{0}^{\ZZ^{2}} \to [0,1]$ with respective supports on boxes $B_{1}, B_{2} \subset \ZZ^{2}$ (in the sense of Equation \eqref{eq:support}),
\begin{equation*}
\EE_{q}[f_{1}f_{2}] \leq \EE_{q}[f_{1}]\EE_{q}[f_{2}]+\usebigconstant{c:decoupling_irc}(\per(B_{1})+\per(B_{2}))e^{-\useconstant{c:decoupling_irc_small}\dist^{\sfrac{1}{8}}},
\end{equation*}
where $\dist$ is the distance between the boxes $B_{1}$ and $B_{2}$. Observe that, if the two boxes have disjoint projections in the $x$-axis, then the functions $f_{1}$ and $f_{2}$ are independent.

\par This implies we can use Theorem \ref{teo:percolation}. It remains to verify condition \eqref{eq:trigger}. We proceed by taking $R_{k}=l_{k}+L_{k}+1$. Notice that we can bound
\begin{equation}\label{eq:trigger_irc}
\begin{split}
l_{k}^{4}p_{k} & \leq l_{k}^{4}\PP_{q} \left[\begin{array}{cl}
[0,l_{k}] \times [0,L_{k}] \text{ does not} \\ \text{have a vertical open crossing} \end{array}\right] \\
& \leq l_{k}^{4}L_{k}(1-q_{0})^{l_{k}}.
\end{split}
\end{equation}
If $k$ is large enough, the quantity above is smaller than one. For such value of $k$, we take $R=R_{k}=l_{k}+L_{k}+1$, and Equation \eqref{eq:trigger} is verified.

\bigskip
\bigskip

\par \textbf{Level sets in independent random walks.} This is a good example of how to apply Theorem \ref{teo:percolation} using condition \eqref{eq:sprinkling_decoupling}. Here we work with the discrete framework. At time zero, each integer site receives independently a $\poisson(\lambda)$ number of particles that move independently as discrete time random walks. We will verify that survival occurs on the level sets $\{\eta_{t}(x) \geq j\}$, for $j \in \NN_{0}$ fixed. Here, the set of open vertices is formed by all points $(x,t) \in \ZZ^{2}$ with at least $j$ particles in site $x$ at time $t$ and the set $S_{R}$ is given in Example \eqref{ex:detection_set}.

\par We use of the decoupling for this dynamics proved in \cite{hhsst}.
\begin{prop}[Corollary 3.1 from \cite{hhsst}]
Let $B_1 = ([a, b] \times [n, m]) \cap \ZZ^2$ and $B_2 = ([a', b'] \times [-n', 0]) \cap \ZZ^2$ be two
space-time boxes and assume that $n \geq c$, for some $c$ large enough. Assume that $f_{1}$ and $f_{2}$ are non-increasing random variables with support
in $B_1$ and $B_2$, respectively, taking values in $[0,1]$. Then, for any $\rho \geq 1$,
\begin{equation}
\EE_{\rho(1+n^{-1/16})}[f_{1} f_{2}] \leq \EE_{\rho(1+n^{-1/16})}[f_{1}] \,\,
\EE_{\rho}[f_2] + C \big(\per(B_{1}) + n\big)\,e^{-C^{-1} n^{1/8}}.
\end{equation}
\end{prop}

\par Now, if we take any two boxes $B_{1}$ and $B_{2}$, satisfying
\begin{equation}
\dist > 6(\per(B_{1})+\per(B_{2}))+c,
\end{equation}
for some constant $c$, we simply observe that, if the boxes are far away in space, the functions are independent. The reason to why this is true is that we are working on discrete time, and then no particle can cross both of them. On the other hand, if they are close in space, they are necessarily far away in time, and we can use the proposition above to conclude that
\begin{equation}\label{eq:decoupling_irw}
\EE_{\rho}[f_{1} f_{2}] \leq \EE_{\rho(1-\tilde{c}\dist^{-\sfrac{1}{16}})}[f_{1}] \,\,
\EE_{\rho(1-\tilde{c}\dist^{-\sfrac{1}{16}})}[f_2] + C \dist\,e^{-C^{-1} \dist^{1/8}},
\end{equation}
for some constants $\tilde{c}$ and $C$.

\par This verifies conditions \eqref{eq:sprinkling_decoupling} and \eqref{eq:error_decay_sprinkling} with $\delta=\sfrac{1}{16}$. Since we need to use this assumption in the proof of Theorem \ref{teo:percolation} with $\epsilon$ given by $l_{k}^{-\delta}$, we observe that $\dist \geq l_{k}$, for the sets taken in Lemma \ref{lemma:renormalization}.

\par Now, if we are given a density $\lambda_{\infty}$ for which we want to verify Theorem \ref{teo:detection}, we define the sequence $\lambda_{k}$ as in \eqref{eq:u_samll}, but with $l_{k}^{-\delta}$ replaced by $\tilde{c}l_{k}^{-\sfrac{1}{16}}$. The verification of the trigger assumption \eqref{eq:trigger} is done in the same way as in \eqref{eq:trigger_irc}.

\par This proves that, for any $\lambda>0$, if $R$ is large enough, survival is possible on top of level $j$.

\bigskip
\bigskip

\par \textbf{Independent random walks} Although Theorem \ref{teo:detection} is already known for this underlying dynamics, see \cite{stauffer} for the case when Brownian motions are considered, here we apply a combination of Remark \ref{remark:vacant_percolation} and condition \eqref{eq:sprinkling_decoupling} to prove that survival on top of the empty sites using the set of paths $S_{R}$ from Example \eqref{ex:detection_set}.

\par As a corollary of Lemma 6.5 from \cite{st} it is easy to obtain
\begin{lemma}
Let $B_1 = ([a, b] \times [n, m]) \cap \ZZ^2$ and $B_2 = ([a', b'] \times [-n', 0]) \cap \ZZ^2$ be two
space-time boxes and assume that $n \geq c$, for some $c$ large enough. Assume that $f_{1}$ and $f_{2}$ are non-decreasing random variables with support
in $B_1$ and $B_2$, respectively, taking values in $[0,1]$. Then, for any $\rho \in [r,R]$ and $\epsilon \in (0,1)$,
\begin{equation}
\EE_{\rho}[f_{1} f_{2}] \leq \EE_{\rho(1+\epsilon)}[f_{1}] \,\,
\EE_{\rho(1+\epsilon)}[f_2] + C \big(\per(B_{1}) + n\big)\,e^{-C^{-1} r \epsilon^{2} n^{1/4}}.
\end{equation}
\end{lemma}

\par From the lemma above, just like in the discussion before Equation \eqref{eq:decoupling_irw}, we conclude the existence of constants $A, B$ and $C$ such that, if
\begin{equation}
\dist = \dist(B_{1}, b_{2}) \geq A \left(\per(B_{1} + \per(B_{2})\right)+B,
\end{equation}
then
\begin{equation}
\EE_{\rho}[f_{1} f_{2}] \leq \EE_{\rho(1+\epsilon)}[f_{1}] \,\,
\EE_{\rho(1+\epsilon)}[f_2] + C \dist \,e^{-C^{-1} r \epsilon^{2} \dist^{1/4}}.
\end{equation}

\par This implies we can apply Therem \ref{teo:percolation}, using Remark \ref{remark:vacant_percolation} and condition \eqref{eq:sprinkling_decoupling}. The verification of \eqref{eq:trigger} is done just like in \eqref{eq:trigger_irc}.

\section{Decoupling for the exclusion process}\label{sec:decoupling}
~
\par This section is devoted to the proof of the exclusion process decoupling. This decoupling will be used to prove the correlation decay need in order to apply Theorem \ref{teo:percolation} in the proof of Theorem \ref{teo:detection}.

\par We split our discussion in three subsections. We begin with a brief review of some facts that will be used in our proof. In the second subsection, the decoupling is proved assuming Lemma \ref{lemma:coupling_ep}. The third subsection is devoted to the proof of Lemma \ref{lemma:coupling_ep}, a coupling that plays a central role in the proof of the decoupling.

\subsection{A brief review of the exclusion process}\label{subsec:ep}
~
The exclusion process $(\eta_{t})_{t \geq 0}$ on $\ZZ$ is a Markov process with state space $\{0,1\}^{\ZZ}$ and generator given by
\begin{equation}\label{eq:generator}
Lf(\eta)=\frac{1}{2}\sum_{x \in \ZZ} \sum_{h=\pm 1}\eta(x)(1-\eta(x+h))\left[f(\eta^{x,x+h})-f(\eta)\right],
\end{equation}
where $f: \{0,1\}^{\ZZ} \to \RR $ is any local function and $ \eta^{x,y} $ is the configuration given by
\begin{equation*}
(\eta^{x,y})(z)=\left\{\begin{array}{cl}
\eta(y),& \mbox{if}\,\,\, z=x,\\ 
\eta(x),& \mbox{if} \,\,\, z=y,\\ 
\eta(z),& \mbox{ otherwise.}
\end{array}
\right .
\end{equation*}

\par For each $ \rho \in [0,1] $, define $\mu_{\rho}$ as the product measure on $\{0,1\}^{\ZZ}$ with marginals given by
\begin{equation}\label{eq:invariant_measures}
\mu_{\rho}\{\eta : \eta(k)=1\}=1-\mu_{\rho}\{\eta: \eta(k)=0\}=\rho, \,\,\,\,\, \text{for all }\, k \in \ZZ.
\end{equation}

\par It is a well known fact that the process $(\eta_{t})_{t \geq 0}$ is reversible with respect to the measure $\mu_{\rho}$. We call the parameter $\rho$ the density of the process if its starting configuration is distributed as $\mu_{\rho}$. Denote by $\PP_{\rho}$ the distribution of the exclusion process $(\eta_{t})_{t \geq 0}$ with density $\rho$.

\par We also recall a classical graphical construction of the exclusion process that will be useful. This construction is made with the help of the interchange process, that we denote by $\gamma$. First consider an independent Poisson process of rate $\sfrac{1}{2}$ for each edge $(x,x+1)$ of $\ZZ$. We will represent the Poisson processes in the edges by arrows (as in Figure \ref{fig:interchange_process}). Observe that for each site $ x \in \ZZ$ and $t \geq 0$ there exists an almost sure unique path that starts at $(x,t)$, ends in $\ZZ \times \{0\}$, goes downwards and is forced to cross all arrows it encounters. We denote the end position of this path by $\gamma_{t}(x) \in \ZZ$, the label of the interchange process in site $x$ at time $t$ (see Figure \ref{fig:interchange_process}).

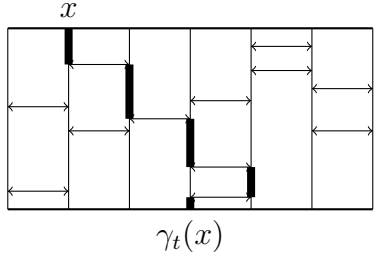
\begin{figure}[h]\label{fig:interchange_process}
\begin{center}
\begin{tikzpicture}[scale=0.8]

% Grids
\draw[-, thick]  (-3,0) -- (3,0);
\draw[-, thick]  (-3,3) -- (3,3);

\foreach \x in {-3,-2,-1,0,1,2,3}
	{
	\draw[-] (\x,0) -- (\x, 3);
	}

% Poisson processes
\foreach \y in {0.3,1.7}
	{
	\draw[<->](-3,\y)--(-2,\y);
	}
\foreach \y in {1.3,2.4}
	{
	\draw[<->] (-2,\y) -- (-1,\y);
	}
\foreach \y in {1.5}
	{
	\draw[<->] (-1,\y) -- (0,\y);
	}
\foreach \y in {0.2,0.7,1.8}
	{
	\draw[<->] (0,\y) -- (1,\y);
	}
\foreach \y in {2.3,2.7}
	{
	\draw[<->] (1,\y) -- (2,\y);
	}
\foreach \y in {2,1.3}
	{
	\draw[<->] (2,\y) -- (3,\y);
	}
	
% Particle evolution
\draw[line width=3pt] (0,0)--(0,0.2);
\draw[line width=3pt] (1,0.2)--(1,0.7);
\draw[line width=3pt] (0,0.7)--(0,1.5);
\draw[line width=3pt] (-1,1.5)--(-1,2.4);
\draw[line width=3pt] (-2,2.4)--(-2,3);

% Labels
% Labels

\node[below] at (0,0) {$\gamma_{t}(x)$};
\node[above] at (-2,3) {$x$};

\end{tikzpicture}
\caption{The points $x$ and $\gamma_{t}(x)$.}
\end{center}
\end{figure}

\par Given an initial configuration $\eta_{0}$ for the exclusion process we obtain the configuration at time $t$ by setting
\begin{equation}\label{eq:graphical_construction}
\eta_{t}(x)=\eta_{0}(\gamma_{t}(x)), \,\,\,\,\, \text{for all } \, x \in \ZZ.
\end{equation}

\par This construction results in the Markov process with generator given by \eqref{eq:generator}. Observe that each particle as well as each hole in $\eta_{0}$ performs a continuous time random walk.

\par The space of configurations $\{0,1\}^{\ZZ}$ has a partial order (similar to the order in $\{0,1\}^{\ZZ^{2}}$ defined in Equation \eqref{eq:partial_order}) given by
\begin{equation}
\eta \preceq \xi \text{ if and only if } \eta(x) \leq \xi(x), \,\,\,\, \text{for all } \,x \in \ZZ.
\end{equation}
If we use the same Poisson clocks in the graphical construction presented above for two different initial conditions, this partial order is preserved.

\par When the starting configuration of the exclusion process is distributed as $\mu_{\rho}$, by reversibility it is possible to use the graphic construction presented above to construct the exclusion process for negative times.

\par Using a coupling of the measures $(\mu_{\rho})_{\rho \in [0,1]}$ that is increasing in the partial order of $\{0,1\}^{\ZZ}$, we can construct in the same probability space all the processes $(\eta_{t}^{\rho})_{t \in \RR, \rho \in [0,1]}$ in a way that
\begin{itemize}
\item[1.] $(\eta^{\rho}_{t})_{t \in \RR}$ is an exclusion process with density $\rho$;
\item[2.] if $\rho \leq \rho '$ then $ \eta_{t}^{\rho} \preceq \eta_{t}^{\rho '}$, for all $ t \in \RR$.
\end{itemize}

\subsection{Decoupling}
~
\par In this subsection we prove the exclusion process decoupling stated in Theorem \ref{teo:decoupling_ep}. In the proof we will assume the existence of the coupling stated in Lemma \ref{lemma:coupling_ep}, a central tool in the proof.

\par We will work with functions defined in the space of trajectories $\mathbf{S}=\text{D}_{\RR}\{0,1\}^{\ZZ}$ of the exclusion process. It will be useful to think of the domain of the functions as the set $\{0,1\}^{\ZZ \times \RR}$.

\par Notice that the correlation decay in \eqref{eq:correlation_decay_percolation} is not true for the exclusion process, as pointed out in the next remark.

%%%%
\newconstant{c:correlation_decay}
%%%%

\begin{remark}\label{remark:bad_correation_decay}
Recall the construction of the exclusion process from the interchange process in \eqref{eq:graphical_construction}. Using the independence between the configuration $\eta_{0}$ and the interchange process $\gamma$, we compute
\begin{align*}
\cov_{\rho}(\eta_{t}(0), \eta_{0}(0)) & = \EE_{\rho}(\eta_{t}(0) \eta_{0}(0))-\rho^{2} \\
& = \EE_{\rho}(\eta_{t}(0) \eta_{0}(0)(\mathbf{1}_{\{\gamma_t(0)=0\}}+\mathbf{1}_{\{\gamma_t(0)\neq 0\}}))-\rho^{2} \\
& = \rho^{2}\PP[\gamma_t(0)\neq 0]+\rho\PP[\gamma_t(0) = 0] - \rho^{2} \\
& = (\rho-\rho^{2})\PP[\gamma_t(0) = 0] \\
& \geq \frac{\useconstant{c:correlation_decay}}{\sqrt{t}},
\end{align*}
which proves that \eqref{eq:correlation_decay_percolation} does not hold. The last inequality is a consequence of the fact that $\gamma_{t}(0)$ has the distribution of a continuous time random walk.
\end{remark}

\par For the proof of Theorem \ref{teo:decoupling_ep}, there are two different cases to take care of: Either the horizontal distance between the boxes is large or the vertical distance is. In the first case, we only need to use some moderate deviation estimates to get the bounds we need. In the second case, we use a coupling between two exclusion process with densities $\rho < \rho'$. This coupling assures us that the process with density $\rho$ is dominated by the process with higher density in an interval $I$ if the time is large enough. The existence of this coupling is the content of the following lemma:

%%%%
\newbigconstant{C:couplingep}
\newconstant{c:couplingep}
%%%%
\begin{lemma}\label{lemma:coupling_ep}
There exist positive constants $\useconstant{c:couplingep}$ and $\usebigconstant{C:couplingep}$ such that the following holds. For $\rho < \rho' \in [0,1]$, any given interval $ I =[c,d] \subset \RR$ with $c,d \in \ZZ$ and time $t \geq \usebigconstant{C:couplingep}$, there exists a coupling $\PP$ of two exclusion process with independent initial conditions $\eta_{0} \sim \mu_{\rho}$ and $\xi_{0} \sim \mu_{\rho'}$ in a way that $(\xi_{s})_{s \geq 0}$ is independent of $\eta_{0}$ and
\begin{displaymath}
\PP \Big[\exists \, x \in I\cap\ZZ : \eta_{t}(x)>\xi_{t}(x)\Big] \leq \useconstant{c:couplingep} t(t+|I|)\exp\left\{-\useconstant{c:couplingep}^{-1}(\rho'-\rho)^{2}t^{\sfrac{1}{4}}\right\}. 
\end{displaymath}
\end{lemma}

\par The proof of this lemma is contained in the next subsection. We now use it to conclude Theorem \ref{teo:decoupling_ep}.

\begin{proof}[Proof of Theorem \ref{teo:decoupling_ep}]
Let $d_{H}$ and $d_{V}$ denote the horizontal and vertical distances between the boxes $B_{1}$ and $B_{2}$:
\begin{equation*}
d_{H}=\inf \{|x-y|:(x,t) \in B_{1} \text{ and } (y,s) \in B_{2}\},
\end{equation*}
and
\begin{equation*}
d_{V}=\inf \{|t-s|:(x,t) \in B_{1} \text{ and } (y,s) \in B_{2}\}.
\end{equation*}

Assume first that
\begin{equation}\label{eq:large_horizontal_distance}
d_{H} \geq 3(\per(B_{1})+\per(B_{2})+d_{V}).
\end{equation}

In this case, observe that, if a particle (or a hole) of the exclusion process touched both boxes, it jumped at least $d_{H}$ times in at most $\per(B_{1})+\per(B_{2})+d_{V}$ units of time. Since these particles (as well as the empty sites) move as random walks, the number of jumps in a given period of time has Poisson distribution. This implies that
\begin{multline}
\PP\left[\begin{array}{cl}
\text{a fixed particle (or hole)}  \\ \text{touches both boxes}\end{array}\right] \\ \leq  \PP\left[\poisson(\per(B_{1})+\per(B_{2})+d_{V}) \geq d_{H}\right]
\leq e^{-\useconstant{c:poisson_concentration_3}(d_{V}+d_{H})},
\end{multline}
with $\useconstant{c:poisson_concentration_3}>0$ given by Lemma \ref{concentration_poisson_3} in the Appendix.

Now we only need to count how many particles can touch both boxes. Fix an arbitrary site between the two boxes. Each particle that touches both boxes must cross this fixed site. This implies that we can bound the number of particles by the number clocks that ring in the neighboring edges of this site. It turns out that this has also Poisson distribution with parameter $\per(B_{1})+\per(B_{2})+d_{V}$. Hence, with probability at least $1-e^{-\useconstant{c:poisson_concentration_3}(d_{V}+d_{H})}$, there are at most $d_{H}$ particles that can cross this site. Using a union bound, we get
\begin{align*}
\PP\left[\begin{array}{cl}
\text{some particle (or hole)}  \\ \text{touches both boxes}\end{array}\right] & \leq  \PP\left[\begin{array}{cl}
\text{more than $d_{H}$ clocks ring in the} \\ \text{neighboring edges of the fixed site}\end{array}\right] \\
& + d_{H}\PP\left[\begin{array}{cl}
\text{a fixed particle (or hole)}  \\ \text{touches both boxes}\end{array}\right] \\
& \leq (1+d_{H})e^{-\useconstant{c:poisson_concentration_3}(d_{V}+d_{H})}.
\end{align*}

Now, if we condition on the trajectories inside $B_{1}$, we can split the expectation below according to the existence of particles that touch both boxes and conclude that
\begin{equation*}
\EE_{\rho}(f_{1}f_{2}) \leq \EE_{\rho}(f_1)\EE_{\rho}(f_2)+(1+d_{H})e^{-\useconstant{c:poisson_concentration_3}(d_{V}+d_{H})},
\end{equation*}
that is stronger than \eqref{eq:decoupling_estimate}.

Assume now that \eqref{eq:large_horizontal_distance} does not hold, i.e., assume that $d_{H} \leq 3(\per(B_{1})+\per(B_{2})+d_{V})$. This, combined with Equation \eqref{eq:distance_hypothesis}, implies that
\begin{equation}\label{eq:time_distance}
d_{V} \geq \frac{\usebigconstant{c:decouplingep}}{4}.
\end{equation}
If we take $\usebigconstant{c:decouplingep} \geq 4\usebigconstant{C:couplingep}$, the equation above will allow us to use Lemma \ref{lemma:coupling_ep}. In this case we use a different approach.

We will assume that the boxes have the form
\begin{align*}
B_{1}=[\tilde{a}, \tilde{b}]\times[-\tilde{t},0], \\
B_{2}=[a,b]\times[t,t+s].
\end{align*}
Figure \ref{fig:boxes} can be used as a reference in this part of the proof.

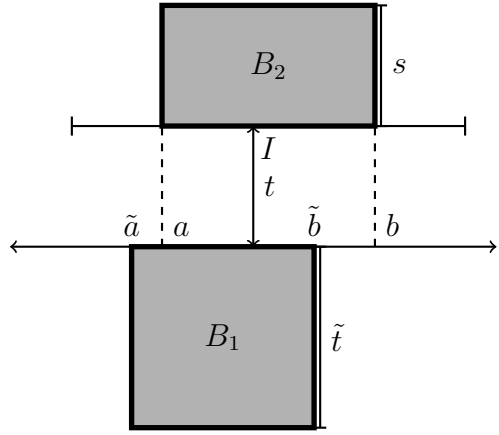
\begin{figure}[h]\label{fig:boxes}
\begin{center}
\begin{tikzpicture}[scale=0.8]

% Grids
\draw[<->, thick]  (-4,0) -- (4,0);

% Boxes
\fill [black!30!](-2,-3) rectangle (1,0);
\draw [line width=2pt] (-2,-3) rectangle (1,0);
\node at (-0.5,-1.5) {$B_{1}$};

\fill [black!30!](-1.5,2) rectangle (2,4);
\draw [line width=2pt] (-1.5,2) rectangle (2,4);
\node at (0.25,3) {$B_{2}$};

% Interval I
\draw[|-|, thick](-3,2)--(3.5,2);
\node[below] at (0.25,2) {$I$};

% Time lines
\draw[<->, thick]  (0,0) -- (0,2);
\node[right] at (0,1) {$t$};

\draw[thick] (2.1,2)--(2.1,4);
\draw[thick] (2,2)--(2.2,2);
\draw[thick] (2,4)--(2.2,4);
\node[right] at (2.1,3) {$s$};

\draw[thick] (1.1,0)--(1.1,-3);
\draw[thick] (1,0)--(1.2,0);
\draw[thick] (1,-3)--(1.2,-3);
\node[right] at (1.1,-1.5) {$\tilde{t}$};

% Space lines

\node[above] at (-2,0) {$\tilde{a}$};
\node[above] at (1,0) {$\tilde{b}$};

\draw[thick, dashed] (-1.5,0)--(-1.5,2);
\node[above right] at (-1.5,0) {$a$};

\draw[thick, dashed] (2,0)--(2,2);
\node[above right] at (2,0) {$b$};

\end{tikzpicture}
\caption{The boxes $B_{1}$, $B_{2}$ and the interval $I$ (used in the proof of Theorem \ref{teo:decoupling_ep}).}
\end{center}
\end{figure}

Let $\mathcal{F}=\sigma\left(\eta_{u};u \leq 0\right)$ and use the Markov property to get
\begin{align}\label{eq:expectation_markov_property}
\EE_{\rho}(f_{1}f_{2})& = \EE_{\rho}(\EE_{\rho}(f_{1}f_{2}|\mathcal{F})) \nonumber \\
& = \EE_{\rho}(f_{1}\EE_{\rho}(f_{2}|\mathcal{F})) \\
& = \EE_{\rho}(f_{1}\EE_{\rho}(f_{2}|\eta_{0})). \nonumber
\end{align}
To estimate the conditional expectation above, we apply Lemma \ref{lemma:coupling_ep} with $I=[a-2s-t,b+2s+t]$ (see Figure \ref{fig:boxes}) and define the event
\begin{displaymath}
A=\left\{\eta\,:\,\begin{array}{cl}
\text{all particles of $\eta$ that pass through} \\ \text{the box $B_{2}$ are inside $I$ at time $t$}\end{array}\right\}.
\end{displaymath}

Now, we split the conditional expectation in estimate \eqref{eq:expectation_markov_property}, use Lemma \ref{lemma:coupling_ep}, and the fact that both functions $f_{1}$ and $f_{2}$ are positive, bounded by one and non-decreasing to get
\begin{align*}
\EE_{\rho}(f_{1}f_{2})& = \EE_{\rho}\Big(f_{1}(\eta)\EE_{\rho}\Big(f_{2}(\eta)\Big|\eta_{0}\Big)\Big) \\
& \leq \EE\Big(f_{1}(\eta)\EE\Big(f_{2}(\xi)\charf{\{\forall \, x \in I\cap\ZZ : \eta_{t}(x) \leq \xi_{t}(x)\}}\charf{\{A\}}\Big|\eta_{0}\Big)\Big) \\
& \quad +\EE\Big(f_{1}(\eta)\EE\Big(f_{2}(\xi)(\charf{\{\exists \, x \in I\cap\ZZ : \eta_{t}(x) > \xi_{t}(x)\}}+\charf{\{A^{c}\}})\Big|\eta_{0}\Big)\Big) \\
& \leq \EE_{\rho}(f_1)\EE_{\rho'}(f_2)+\PP\left[A^{c}\right] +\PP\left[\exists \, x \in I\cap\ZZ : \eta_{t}(x)>\xi_{t}(x)\right].
\end{align*}

To bound the last probability above we use Lemma \ref{lemma:coupling_ep}. Now all we need to do is estimate the probability of $A^{c}$. We use a similar argument to the one used in the first part of the proof. Begin by observing that in order for a particle that is at the $k$-th site at the right of $I$ at time $t$ to enter the box $B_{2}$ it is necessary that it jumps more than $2s+t+k$ times before time $t+s$. We also know that the number of jumps of a particle during the time interval $[t,t+s]$ has distribution Poisson distribution with parameter $s$. This allows us to estimate
\begin{displaymath}
\PP[A^{c}] \leq 2\sum_{k=0}^{+\infty}\PP[\poisson(s) \geq 2s+t+k] \leq 2\sum_{k=0}^{+\infty}e^{-\useconstant{c:poisson_concentration_1}(t+k)} = c e^{-\useconstant{c:poisson_concentration_1} t},
\end{displaymath}
where in the second inequality we used a simple large deviation estimate given by Lemma \ref{concentration_poisson} in the Appendix.

If we put all this together we get:
\begin{align}\label{eq:final_estimate_decoupling}
\EE_{\rho}(f_{1}f_{2})  & \leq \EE_{\rho}(f_1)\EE_{\rho'}(f_2)+c e^{-\useconstant{c:poisson_concentration_1} t} \nonumber \\
& +\useconstant{c:couplingep} t(t+|I|)\exp\left\{-\useconstant{c:couplingep}^{-1}(\rho'-\rho)^{2}t^{\sfrac{1}{4}}\right\} \\
& \leq \EE_{\rho'}(f_1)\EE_{\rho'}(f_2)+ \useconstant{c:couplingep} t(t+|I|)\exp\left\{-\useconstant{c:couplingep}^{-1}(\rho'-\rho)^{2}t^{\sfrac{1}{4}}\right\}, \nonumber
\end{align}
by possibly changing the constants in the last estimate.

Now, since $t=d_{V}$ and $d_{H} \leq 3(\per(B_{1})+\per(B_{2})+d_{V})$, we have
\begin{align*}
\dist & \leq \sqrt{2}(d_{H} + d_{V}) \leq \sqrt{2}\left(4d_{V}+3(\per(B_{1})+\per(B_{2})\right) \leq \sqrt{2}\left(\frac{\dist}{2}+4d_{V}\right),
\end{align*}
and hence $\dist \leq 4\left(1-\frac{\sqrt{2}}{2}\right)^{-1} d_{V}$. Substituting this on estimate \eqref{eq:final_estimate_decoupling} concludes the proof.
\end{proof}

\subsection{Coupling}\label{subsec:coupling_ep}
~
\par In this subsection we construct the coupling of Lemma \ref{lemma:coupling_ep}. We begin by giving an informal description of the coupling and then we make all the estimates need to get the domination.

\par Consider two independent initial configurations $\eta_{0} \sim \mu_{\rho}$ and $\xi_{0} \sim \mu_{\rho'}$ with $\rho<\rho'$. In our coupling, we want to obtain domination in an interval $I$ for a large time $t$. Due to the bounded velocity that the particles have, we only need to look at particles that at time zero are inside a sufficiently large interval $H$ that contains $I$.  Once we have a bound on the probability that some particle spends time outside $H$ and is inside $I$ at time $t$ we can restrict ourselves to particles that stay inside the interval $H$ for all times before time $t$.

\par We know that each particle, in both processes, performs a random walk. We want our coupling to behave in a way that if two particles of different process are in the same site of $H$ at some time $s \leq t$, then they move together from this time on.

\par With this greedy strategy it is not possible to get good bounds on the probabilities we need. To get around this problem, at time zero we will match the particles in pairs that will stay together if they meet.

\par We would like that these particles do not take a long time to do so and to control this we need to assure that they are close at time zero. Therefore, we introduce a partition of the interval $H$ with intervals $(I_{j})_{j=1}^{N}$ of controlled length. Due to the difference of densities, we expect that with high probability each of the intervals $I_{j}$ has more particles of the configuration $\xi_{0}$ than particles of $\eta_{0}$. When this happens, we can match all the particles of $\eta_{0}$ to some particle of $\xi_{0}$ in a way that they belong to the same interval of the covering at time zero.

\par Once we have the couples at time zero we need to set the evolution. We will make use of two independent copies of the graphical construction of the exclusion process presented in Subsection \ref{subsec:ep}. We make the process $\xi$ follow one of them and the evolution of the process $\eta$ will alternate between the two graphical constructions in order to get the property that coupled particles stay together.

\par We can get bounds on the probability that two particles do not meet up to time $t$, but the decay is not as good as the one in Lemma \ref{lemma:coupling_ep}. To get the desired bound we have to repeat the same procedure more than one time. So we split the time interval $[0,t]$ into smaller intervals $[t_{i-1},t_{i})$, where $0=t_{0} < t_{1}< \dots <t_{k}=t$, and set the evolution on these intervals. When we reach the end point $t_{i}$ we take another matching (this is done in a way that couples that already met stay together) and let the system evolve once again. This will allow us to get stretched exponential bounds as claimed in Lemma \ref{lemma:coupling_ep}.

\bigskip

\par Now we present the rigorous construction of this coupling. First we introduce the intervals we will consider and the matching that we need. The second step is to set the evolution of the coupled process and the last step is to repeat this procedure.

\par Given the interval $I=[a,b]$, we define $H=[a-\lceil 3t \rceil,b+\lceil 3t \rceil]$ and cover it with disjoint subintervals with length $L=\lfloor t^{\sfrac{1}{4}} \rfloor$. Let us call $(I_{j})_{j=1}^{N}$ these intervals. Observe that we have at most $|H|$ intervals in the covering. Figure \ref{fig:intervals} can be used to keep track of the notation. It may be necessary to increase the size of $H$ to make sure that all the intervals of the covering $(I_{j})_{j=1}^{N}$ have exactly $L$ integer points. Notice that we need to increase the size of $H$ by at most $L$.

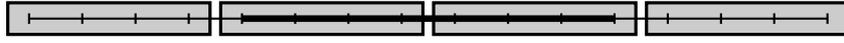
\begin{figure}[h]\label{fig:intervals}
\centering
\begin{tikzpicture}[scale=0.7]

% Subntervals
\foreach \x in {-7.5,-3.5,0.5,4.5}
	{
	\fill[black!20!] (\x-0.4,-0.3) rectangle (\x+3.4,0.3);
	\draw[very thick](\x-0.4,-0.3) rectangle (\x+3.4,0.3);
	}

% Grids
\draw[thick] (-7.5,0)--(7.5,0);
\foreach \x in {-7.5,-6.5,-5.5,-4.5,-3.5,-2.5,-1.5,-0.5,0.5,1.5,2.5,3.5,4.5,5.5,6.5,7.5}
	{
	\draw[thick](\x,-0.1)--(\x,0.1);
	}
% Initial interval
\draw[line width=2.5](-3.5,0)--(3.5,0);
\end{tikzpicture}
\caption{The grid is the interval $H$, the thicker line is the interval $I$ and the covering $(I_{j})_{j=1}^{N}$ is represented by the gray rectangles.}
\end{figure}

\par We want to match particles that are inside the same interval of the partition $(I_{j})_{j=1}^{N}$. It is necessary to control the number of particles inside each one of these intervals for the given configurations. This leads us to define $\sum_{x \in I_{j}} \eta(x)=\sigma_{j}(\eta)$,  the number of particles inside the interval $I_{j}$ for the configuration $\eta$.

\begin{claim} If $\eta_{0} \sim \mu_{\rho} $, $\xi_{0} \sim \mu_{\rho'}$ are initial configurations with $ \rho < \rho'$ and $\bar{\rho}=\sfrac{1}{2}(\rho+\rho')$ then
\begin{displaymath}
\PP \left[\min\{\sigma_{j}(\xi_{0})\} \leq \bar{\rho}L\right] \leq |H|\exp\left\{-\frac{L(\rho'-\rho)^{2}}{8}\right\},
\end{displaymath}
and
\begin{displaymath}
\PP\left[\max\{\sigma_{j}(\eta_{0})\} \geq \bar{\rho}L \right] \leq |H|\exp\left\{-\frac{L(\rho'-\rho)^{2}}{8}\right\}.
\end{displaymath}
\end{claim}
\begin{proof}
Since the invariant measures are product measures, the number of particles in a given interval has the distribution of a sum of i.i.d. random variables that assume only value 0 and 1. This claim is a consequence of a simple large deviation bound, see Corollary \ref{binconcentration} in Appendix.
\end{proof}

\begin{remark}\label{remark:concentration}
\par Notice that the last claim implies
\begin{multline}\label{eq:concentration_initial_configuration}
\PP\left[\exists \, j \leq N : \sigma_{j}(\eta_{0}) \geq \sigma_{j}(\xi_{0}) \right] \leq \PP\left[\min\{\sigma_{j}(\xi_{0})\} \leq \bar{\rho}L \right] \\
+ \PP\left[\exists j \leq N : \sigma_{j}(\eta_{0}) \geq \sigma_{j}(\xi_{0}) \geq \bar{\rho}L \right]
 \leq 2|H|\exp\left\{-\frac{L(\rho'-\rho)^{2}}{8}\right\}.
\end{multline}
It is really important to notice also that in the estimate above we do not need to assume independence between the configurations $\eta_{0}$ and $\xi_{0}$.
\end{remark}

\par When two configurations $(\eta,\xi)$ are not in the event above, we call it a \emph{good pair of configurations} and denote this by $ \eta \preceq_{I} \xi $. In a good pair of configurations, the matching is possible.

\par This matching must satisfy two important properties. The first condition is that if two particles are in the same site, they are paired. The other property we need is that two matched particles are in the same interval of the partition $(I_{j})_{j=1}^{N}$.

\par Suppose we are given a pair $(\eta, \xi)$ of good configurations. It is easy to construct a deterministic pairing of the particles inside each of the intervals $(I_{j})_{j=1}^{N}$ satisfying the properties listed above. We fix from now on any deterministic construction. Figure \ref{fig:exclusion_pairing} shows an example of a matching between two configurations.

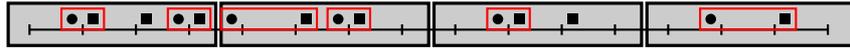
\begin{figure}[h]\label{fig:exclusion_pairing}
\centering
\begin{tikzpicture}[scale=0.7]

% Subntervals
\foreach \x in {-7.5,-3.5,0.5,4.5}
	{
	\fill[black!20!] (\x-0.4,-0.3) rectangle (\x+3.5,0.5);
	\draw[very thick](\x-0.4,-0.3) rectangle (\x+3.5,0.5);
	}

% Grids
\draw[thick] (-7.5,0)--(7.5,0);
\foreach \x in {-7.5,-6.5,-5.5,-4.5,-3.5,-2.5,-1.5,-0.5,0.5,1.5,2.5,3.5,4.5,5.5,6.5,7.5}
		\draw[thick](\x,-0.1)--(\x,0.1);

% Particles
\foreach \x in {-6.5,-4.5,-1.5,-3.5,1.5,5.5}
	{
	\fill[black] (\x-0.2,0.2) circle (0.1);
	}

\foreach \x in {-6.5,-5.5,-4.5,-2.5,-1.5,1.5,2.5,6.5}
	{
	\fill[black] (\x+0.1,0.1) rectangle (\x+0.3,0.3);
	}
	
% Pairing
\foreach \x in {-6.5,-4.5,-1.5,1.5}
	{
	\draw[red, thick] (\x-0.4,0.4) rectangle (\x+0.4, 0);
	}
\draw[red, thick] (5.1,0.4) rectangle (6.9,0);
\draw[red, thick] (-3.9,0.4) rectangle (-2.1,0);

\end{tikzpicture}
\caption{A matching of two configurations. Balls represent the process $\eta$ and squares represent the configuration $\xi$.}
\end{figure}

\par Now that we have the matching, it is possible to set the evolution in our coupling. Keep in mind that we start with two independent configurations $\eta_{0} \sim \mu_{\rho}$ and $\xi_{0} \sim \mu_{\rho'}$ on $\ZZ$, with $ \rho < \rho'$. We need auxiliary random variables for the evolution: Consider two families of independent Poisson processes $(N_{t}^{x,i})_{t \geq0, x \in \ZZ, i=1,2}$ with rate $\sfrac{1}{2}$. Assume also that the Poisson processes are independent of the configurations $\eta_{0}$ and $\xi_{0}$.

\par For the process $\xi$, associate each edge $(x,x+1)$ with the Poisson process $(N_{t}^{x,2})_{t \geq 0}$ and use the graphical construction given in \eqref{eq:graphical_construction} of Subsection \ref{subsec:ep}. The matching is used to evolve the process $\eta$. If $(\eta_{0},\xi_{0})$ is not a good pair of configurations, we use $(N_{t}^{x,1})_{t \geq0, x \in \ZZ}$ for $\eta$ in the same way we did with the process $\xi$. Suppose now that $ \eta_{0} \preceq_{I} \xi_{0} $ and fix an edge $(x,x+1)$. The occupations for the process $\eta$ in the sites $x$ and $x+1$ are exchanged according to $N^{x,2}$ if one of these sites has a pair of matched particles. Otherwise, the occupation changes in these sites for the process $\eta$ obey the exponential times given by $N^{x,1}$. Figure \ref{fig:coupling_evolution} presents some examples of evoutions.

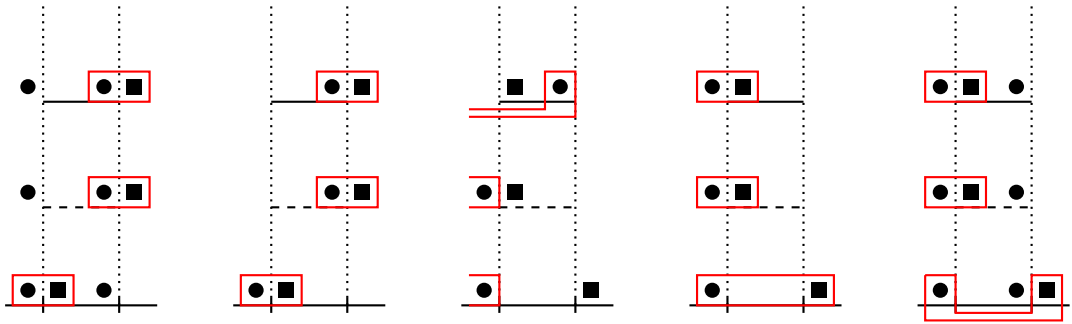
\begin{figure}[h]\label{fig:coupling_evolution}
\begin{center}
\begin{tikzpicture}

% Grids
\draw[thick] (-1,0)--(1,0);
\foreach \x in {-0.5,0.5}
	{
	\draw[thick](\x,-0.1)--(\x,0.1);
	}

% Time lines
\draw[dotted, thick](-0.5,0)--(-0.5,4);
\draw[dotted, thick](0.5,0)--(0.5,4);

% Clocks
\draw[dashed, thick](-0.5,1.3)--(0.5,1.3);
\draw[thick](-0.5,2.7)--(0.5,2.7);

% Particles time zero
\foreach \x in {-0.5,0.5}
	{
	\fill[black] (\x-0.2,0.2) circle (0.1);
	}

\fill[black] (-0.4,0.1) rectangle (-0.2,0.3);
	
% Pairing time zero
	\draw[red, thick] (-0.9,0.4) rectangle (-0.1, 0);

% Particles first change
\foreach \x in {-0.5,0.5}
	{
	\fill[black] (\x-0.2,1.5) circle (0.1);
	}

\fill[black] (0.6,1.4) rectangle (0.8,1.6);

% Pairing first change
\draw[red, thick] (0.1,1.7) rectangle (0.9,1.3);

% Particles second change
\foreach \x in {-0.5,0.5}
	{
	\fill[black] (\x-0.2,2.9) circle (0.1);
	}

\fill[black] (0.6,2.8) rectangle (0.8,3);

% Pairing second change
\draw[red, thick] (0.1,3.1) rectangle (0.9, 2.7);

%Second figure
\begin{scope}[shift={(3,0)}]
% Grids
\draw[thick] (-1,0)--(1,0);
\foreach \x in {-0.5,0.5}
	{
	\draw[thick](\x,-0.1)--(\x,0.1);
	}

% Time lines
\draw[dotted, thick](-0.5,0)--(-0.5,4);
\draw[dotted, thick](0.5,0)--(0.5,4);

% Clocks
\draw[dashed, thick](-0.5,1.3)--(0.5,1.3);
\draw[thick](-0.5,2.7)--(0.5,2.7);

% Particles time zero
\fill[black] (-0.7,0.2) circle (0.1);
\fill[black] (-0.4,0.1) rectangle (-0.2,0.3);
	
% Pairing time zero
	\draw[red, thick] (-0.9,0.4) rectangle (-0.1, 0);

% Particles first change
\fill[black] (0.3,1.5) circle (0.1);
\fill[black] (0.6,1.4) rectangle (0.8,1.6);

% Pairing first change
\draw[red, thick] (0.1,1.7) rectangle (0.9,1.3);

% Particles second change
\fill[black] (0.3,2.9) circle (0.1);
\fill[black] (0.6,2.8) rectangle (0.8,3);

% Pairing second change
\draw[red, thick] (0.1,3.1) rectangle (0.9, 2.7);

\end{scope}

% Third figure
\begin{scope}[shift={(6,0)}]

% Grids
\draw[thick] (-1,0)--(1,0);
\foreach \x in {-0.5,0.5}
	{
	\draw[thick](\x,-0.1)--(\x,0.1);
	}

% Time lines
\draw[dotted, thick](-0.5,0)--(-0.5,4);
\draw[dotted, thick](0.5,0)--(0.5,4);

% Clocks
\draw[dashed, thick](-0.5,1.3)--(0.5,1.3);
\draw[thick](-0.5,2.7)--(0.5,2.7);

% Particles time zero
\fill[black] (-0.7,0.2) circle (0.1);
\fill[black] (0.6,0.1) rectangle (0.8,0.3);

% Pairing time zero
\draw[red, thick] (-0.9,0.4) -- (-0.5, 0.4)--(-0.5,0)--(-0.9,0);

% Particles first change
\fill[black] (-0.7,1.5) circle (0.1);
\fill[black] (-0.4,1.4) rectangle (-0.2,1.6);
	
% Pairing first change
\draw[red, thick] (-0.9,1.7) -- (-0.5, 1.7) -- (-0.5,1.3) -- (-0.9, 1.3);

% Particles second change
\fill[black] (0.3,2.9) circle (0.1);
\fill[black] (-0.4,2.8) rectangle (-0.2,3);
	
% Pairing second change
\draw[red, thick] (-0.9,2.6) -- (0.1,2.6) -- (0.1,3.1) -- (0.5, 3.1) -- (0.5,2.5) -- (-0.9,2.5);

 \end{scope}

% Fourth figure
\begin{scope}[shift={(9,0)}]

% Grids
\draw[thick] (-1,0)--(1,0);
\foreach \x in {-0.5,0.5}
	{
	\draw[thick](\x,-0.1)--(\x,0.1);
	}

% Time lines
\draw[dotted, thick](-0.5,0)--(-0.5,4);
\draw[dotted, thick](0.5,0)--(0.5,4);

% Clocks
\draw[dashed, thick](-0.5,1.3)--(0.5,1.3);
\draw[thick](-0.5,2.7)--(0.5,2.7);

% Particles time zero
\fill[black] (-0.7,0.2) circle (0.1);
\fill[black] (0.6,0.1) rectangle (0.8,0.3);

% Pairing time zero
\draw[red, thick] (-0.9,0.4) rectangle (0.9,0);

% Particles first change
\fill[black] (-0.7,1.5) circle (0.1);
\fill[black] (-0.4,1.4) rectangle (-0.2,1.6);
	
% Pairing first change
\draw[red, thick] (-0.9,1.7) rectangle (-0.1, 1.3);

% Particles second change
\fill[black] (-0.7,2.9) circle (0.1);
\fill[black] (-0.4,2.8) rectangle (-0.2,3);
	
% Pairing second change
\draw[red, thick] (-0.9,3.1) rectangle (-0.1, 2.7);

 \end{scope}

% Fifth figure
\begin{scope}[shift={(12,0)}]

% Grids
\draw[thick] (-1,0)--(1,0);
\foreach \x in {-0.5,0.5}
	{
	\draw[thick](\x,-0.1)--(\x,0.1);
	}

% Time lines
\draw[dotted, thick](-0.5,0)--(-0.5,4);
\draw[dotted, thick](0.5,0)--(0.5,4);

% Clocks
\draw[dashed, thick](-0.5,1.3)--(0.5,1.3);
\draw[thick](-0.5,2.7)--(0.5,2.7);

% Particles time zero
\foreach \x in {-0.5,0.5}
	{
	\fill[black] (\x-0.2,0.2) circle (0.1);
	}

\fill[black] (0.6,0.1) rectangle (0.8,0.3);

% Pairing time zero
\draw[red, thick] (-0.9,0.4) -- (-0.5,0.4) -- (-0.5,-0.1) -- (0.5,-0.1) -- (0.5,0.4) -- (0.9, 0.4) -- (0.9,-0.2) -- (-0.9,-0.2) -- (-0.9,0.4);

% Particles first change
\foreach \x in {-0.5,0.5}
	{
	\fill[black] (\x-0.2,1.5) circle (0.1);
	}
\fill[black] (-0.4,1.4) rectangle (-0.2,1.6);
	
% Pairing first change
\draw[red, thick] (-0.9,1.7) rectangle (-0.1, 1.3);

% Particles second change
\foreach \x in {-0.5,0.5}
	{
	\fill[black] (\x-0.2,2.9) circle (0.1);
	}
\fill[black] (-0.4,2.8) rectangle (-0.2,3);
	
% Pairing second change
\draw[red, thick] (-0.9,3.1) rectangle (-0.1, 2.7);

 \end{scope}

\end{tikzpicture}
\caption{Some cases of the evolution in our coupling. We use the same conventions of Figure \ref{fig:exclusion_pairing}. The Poisson process $N^{1}$ is represented by the lines and $N^{2}$ is represented by the dashed lines.}
\end{center}
\end{figure}

\par By construction, $\xi$ performs an exclusion process. We need to see that the same happens for the process $\eta$.

\begin{claim} The process $\eta$ in this coupling is an exclusion process.
\end{claim}

\begin{proof}
We will prove that up to time $T>0$ the process $\eta$ is an exclusion process. Notice that, by Borel Cantelli Lemma, there exist infinitely many edges that do not ring for neither type up to time $T$. This implies that we can split the integer lattice into intervals that do not exchange particles until time $T$. It is not hard to see that in each one of these intervals the process $\eta$ behaves like an exclusion process: Simply wait until the first of the clocks we are using rings and then, if necessary, update the clocks to mark the next interchange time, according to the coupling. This is exactly an exclusion process in a finite set. This observation implies the claim.
\end{proof}

\par There are some features about this construction that are important to mention. Observe that, for positive times, it is possible to have two particles on the same site (one from each process) that are not matched. The second observation is that the process $\xi$ is independent of $\eta_{0}$, since its evolution depends only on the Poisson processes and its own initial condition.

\par Finally, observe that the distance between a pair of matched particles (that we call $(Z_{s})_{s \geq 0}$) follows the law of a continuous time symmetric random walk $(X_{s})_{s \geq 0}$ sped up by a factor of 2 that dies when it reaches the origin. Besides, since the matched particles lie in the same interval of the partition $(I_{j})_{j=1}^{N}$, the initial position of $Z_{s}$ is at most $L$.Using the reflection principle for random walks and also the heat kernel estimates presented in Appendix \ref{ap_heatkernel} we obtain:
\begin{align}\label{eq:meeting_in_one_step}
\PP & \left[ \begin{array}{cl}
\text{a fixed pair matched of particles} \\ \text{do not meet before time $t$}
\end{array}
\right] \leq \max_{0 \leq k \leq L}  \PP_{k}\left[\inf_{u\leq t}Z_{u} >0 \right] \nonumber \\
& \leq \max_{0 \leq k \leq L}  \PP_{k}\left[\inf_{u\leq 2t}X_{u} >0 \right] = \max_{0 \leq k \leq L} \PP_{0}\left[\sup_{u\leq 2t}X_{u} <k \right] \\
& = \PP_{0}\left[\sup_{u\leq 2t}X_{u} <L \right] = 1-\PP_{0}\left[\sup_{u\leq 2t}X_{u} \geq L \right] \nonumber \\
& \leq 1-2\PP_{0}\left[X_{2t} > L \right] = \PP_{0}\left[|X_{2t}| \leq L \right] \nonumber \\
& = \sum_{k=-L}^{L}\PP_{0}\left[X_{2t} =k \right] \leq \frac{\useconstant{c:continuous_heat_kernel}(2L+1)}{\sqrt{2t}} \leq t^{\sfrac{-1}{8}}, \nonumber
\end{align}
if $t$ is large enough, since $L=\lfloor t^{\sfrac{1}{4}} \rfloor$.

\par The decay obtained in the last estimate is not good enough to get the bounds we need in the error term. To improve this, we change the pairs at some fixed times, obeying the same matching rule. This implies that the particles that already met remain together and give a new chance for those that did not meet their pair yet.

\par Let the \emph{coupling times} be the sequence $\{kt^{\sfrac{3}{4}}\}_{k=1}^{\lfloor t^{\sfrac{1}{4}} \rfloor}$. At these times, we remake the pairing and continue the evolution as explained before. Notice that if a particle has met its couple before some coupling time, then in the new pairing, this particle receives the same partner, since they are in the same site.

\par Let us now list all the possible ways that domination might fail to hold. First, since all the pairings are made inside the interval $H$, we must consider the case where some particle of the process $\eta$ spends time outside $H$ and at time $t$ is inside the interval $I$. To bound this probability, we can simply observe that the endpoints of the interval $H$ are at linear distance from the interval $I$ and use concentration on the number of particles that can make such journey.

\par Once we know all the particles remain inside $H$ all the time up to time $t$, we look at the coupling times. At these times, the matching is remade. Hence, if the configurations are not good for any of them, our coupling fails. To bound this probability we will make use of Remark \ref{remark:concentration}.

\par Now, if we assure also that in all coupling times the configurations are good, the only possibility is that a particle of the process $\eta$ does not find its couple in any of its allowed attempts. With the aid of \eqref{eq:meeting_in_one_step} we can bound this last probability. Our task now is to estimate the probability of all events described above.

\bigskip

\par We begin by setting
\begin{equation}\label{eq:excursion_outside_event}
A=\left\{\begin{array}{cl}
 \text{there are particles of the process $\eta$ that spend time} \\ \text{outside $H$ before time $t$ and are inside $I$ at time t}\end{array}\right\}.
\end{equation}

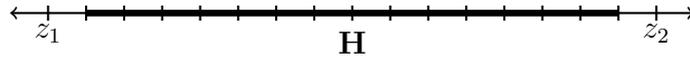
\begin{figure}[h]\label{fig:points}
\begin{center}
\begin{tikzpicture}

% Grids
\draw[thick, <->] (-4.5,0)--(4.5,0);
\foreach \x in {-4,-3.5,-3,-2.5,-2,-1.5,-1,-0.5,0,0.5,1,1.5,2,2.5,3,3.5,4}
	{
	\draw[thick](\x,-0.1)--(\x,0.1);
	}
% Initial interval
\draw[line width=2.5](-3.5,0)--(3.5,0);

% Labels
\node[below] at (-4,0) {$z_{1}$};
\node[below] at (4,0) {$z_{2}$};
\node[below] at (0,-0.1) {\textbf{H}};
\end{tikzpicture}
\caption{The interval H and the points $z_{1}$ and $z_{2}$.}
\end{center}
\end{figure}

\par Let $z_{1}$ be the rightmost site at the left of $H$ and $z_{2}$ the leftmost site at the right of $H$ (see Figure \ref{fig:points}). If a particle makes an excursion outside $H$, it must necessarily pass through $z_{1}$ or $z_{2}$. Since the number of clocks that ring up to time $t$ in the neighboring edges of each of these points is a random variable with Poisson distribution with parameter $t$, we have a good control on the number of particles that spend some time in $z_{1}$. Now, if a particle passes through $z_{1}$ (or $z_{2}$, by symmetry) and at time $t$ it is in the interval $I$, it necessarily jumped at least $\lceil 3t \rceil$ steps in time at most $t$. Since we also know that all particles evolve as random walks, we conclude that in time $t$, the number of jumps that a given particle performs is distributed as Poisson with parameter $t$. Since the particle has at most time $t$ to travel from outside $H$ to $I$ we can estimate:

\begin{align}\label{eq:excursion_outside}
\PP [A] & \leq 2\PP\left[\begin{array}{cl}
\text{the number of clocks in the neighbouring edges of $z_{1}$} \\ \text{that ring before time $t$ is bigger than $3t$}\end{array}\right] \nonumber \\
& +2\PP\left[\begin{array}{cl}
\text{there are at most $3t$ particles that passes through $z_{1}$} \\ \text{before time $t$ and at least one of them is inside $I$ at time $t$ }
\end{array}\right]  \\
& \leq 2\Big(\PP[\text{Po}(t)>3t]+3t\PP[\text{Po}(t)>3t]\Big) \nonumber \\
& \leq 2(e^{-t}+3te^{-t}) \leq (6t+2)e^{-t}. \nonumber
\end{align}

\par Now we focus on the second probability we need to bound. Define
\begin{equation}\label{eq:coupling_times}
B=\left\{\begin{array}{cl}
\text{there exists } k \leq \lfloor t^{\sfrac{1}{4}} \rfloor  \text{ such that} \\ \text{the pair } (\eta_{kt^{\sfrac{3}{4}}},\xi_{kt^{\sfrac{3}{4}}}) \text{ is not good}\end{array}\right\}.
\end{equation}
Notice that the initial law is invariant as pointed out in \eqref{eq:invariant_measures}, but the configurations $(\eta_{kt^{\sfrac{3}{4}}},\xi_{kt^{\sfrac{3}{4}}})$ are not independent. Combining union bounds with estimate \eqref{eq:concentration_initial_configuration}, that also holds for non-independent configurations, gives us
\begin{equation}\label{eq:good_coupling_times}
\PP[B] \leq 2t|H| \exp\left\{-\frac{L(\rho'-\rho)^{2}}{8}\right\}.
\end{equation}

\par Assume now that we are on the event $A^{c} \cap B^{c}$ and let
\begin{equation}\label{eq:couple_meeting}
C=\{\text{there exists } x \in I\cap\ZZ \text{ such that } \eta_{t}(x)>\xi_{t}(x)\}.
\end{equation}
In order for $C \cap A^{c} \cap B^{c}$ to hold, it is necessary that in all attempts, a particle fails to meet its couple. Since each attempt takes time $t^{\sfrac{3}{4}}$, we can use the same computations of estimate \eqref{eq:meeting_in_one_step} (notice that the value of $L$ does not change) to get
\begin{equation}\label{eq:meeting_in_one_step_2}
\PP \left[ \begin{array}{cl}
\text{a fixed pair of particles} \\ \text{do not meet before time $t^{\sfrac{3}{4}}$} \end{array} \right] \leq t^{-\sfrac{1}{16}}.
\end{equation}
To obtain better bounds we use the fact that the matching is remade. We can use union bounds and the fact that our coupling is markovian to obtain
\begin{align}\label{eq:never_meeting}
\PP \left[C \cap A^{c} \cap B^{c}\right] & \leq |H|\sup_{x \in H} \left\{\PP\left[ \begin{array}{cl}
\text{a fixed particle of the process $\eta$ that starts at $x$} \\ \text{does not find any of its couples before time $t$, } A^{c} \cap B^{c} \end{array}\right]\right\}  \nonumber \\
& \leq |H|t^{\frac{-\lfloor t^{\sfrac{1}{4}}\rfloor}{16}} \leq |H|\exp\left\{-\frac{t^{\sfrac{1}{4}}}{32}\log t\right\}.
\end{align}

\par Recall the events \eqref{eq:excursion_outside_event}, \eqref{eq:coupling_times} and \eqref{eq:couple_meeting}. We use estimates \eqref{eq:excursion_outside}, \eqref{eq:good_coupling_times} and \eqref{eq:never_meeting} to get our final bound
\begin{multline}\label{eq:final_estimate_2}
\PP \left[\exists \, x \in I\cap\ZZ : \eta_{t}(x)>\xi_{t}(x)\right] \leq \PP\left[C \cap A^{c} \cap B^{c}\right]+\PP[A]+\PP[B] \\
\leq (6t+2)e^{-t} + 2t|H| \exp\left\{-\frac{L(\rho'-\rho)^{2}}{8}\right\} +|H|\exp\left\{-\frac{t^{\sfrac{1}{4}}}{32}\log t\right\}.
\end{multline}

We can further simplify Equation \eqref{eq:final_estimate_2} by increasing if necessary the value of $\usebigconstant{C:couplingep}$ and get
\begin{align*}
\PP \Big[\exists \, x \in I\cap\ZZ : \eta_{t}(x)>\xi_{t}(x)\Big] \leq \useconstant{c:couplingep} t(t+|I|)\exp\left\{-\useconstant{c:couplingep}^{-1}(\rho'-\rho)^{2}t^{\sfrac{1}{4}}\right\},
\end{align*}
which concludes the proof of Lemma \ref{lemma:coupling_ep}.

\section{Detection}\label{sec:detection}
~
\par Here we use all the tools constructed so far to conclude the proof of Theorem \ref{teo:detection}. The first step is to modify the problem to fit the hypothesis in our percolation model.  The decoupling in the exclusion process will be used to verify the decay correlation on our percolation model.

\par We begin by simply observing that, since the empty spaces of the exclusion process with density $\rho$ also perform an exclusion process with density $1-\rho$, we can prove that it is possible for our target to stay always on top of the exclusion process. This is what we will prove here.

\par  Suppose we constructed in the same probability space the collection $(\eta_{t}^{\rho})_{t \in \RR, \rho \in [0,1]}$ of exclusion processes  with all possible densities in a way that if $\rho \leq \rho'$, then $\eta^{\rho}_{t} \preceq \eta^{\rho'}_{t}$ for all real times $t$.

\par We will construct the family of sets $(\mathcal{I}_{\rho})_{\rho \in [0,1]}$ as described above Equation \eqref{eq:sprinkling_decoupling}. We say that a point $(x,t) \in \ZZ^{2}$ is closed for the density $\rho$ if there exists some time $s \in [t,t+1)$ such that $\eta_{s}^{\rho}(x)=0$. A point is open if it is not closed. Define the set $\mathcal{I}_{\rho}$ as the collection of open points for the density $\rho$. This set is exactly the places where our target is  not detected by a hole of the exclusion process $\eta^{\rho}$ for a period of time of size one. An important observation is that if there exists $ \tilde{g} \in \tilde{S}_{R}$ (see Example \ref{ex:detection_set}) such that
\begin{equation}\label{eq:percolation_exclusion_process}
\text{Range}(\tilde{g}) \subset \mathcal{I}_{\rho},
\end{equation}
then the projection $g$ on the first coordinate axis of $\tilde{g}$ satisfies
\begin{equation}
\eta_{t}(g(\lfloor t \rfloor))=1, \text{ for all } t \in \RR_{+}.
\end{equation}
This implies that non-detection is equivalent to percolation of the set $\mathcal{I}_{\rho}$ using the set of paths given by Example \ref{ex:detection_set}.

\par Let us verify that the sets $(\mathcal{I}_{\rho})_{\rho \in [0,1]}$ satisfy all the necessary hypothesis. First observe that these sets have a translation invariant distribution, since the same is true for the exclusion process. The decay correlation in \eqref{eq:sprinkling_decoupling} is a direct consequence of Theorem \ref{teo:decoupling_ep}, the decoupling for the exclusion process. Fix $\rho_{\infty}>0$ and define $(\rho_{k})_{k \geq 0}$ as in \eqref{eq:u_samll}. Observe that $\rho_{0}>0$

\par Hence, to conclude Theorem \ref{teo:detection}, it is suffice to verify \eqref{eq:trigger} for some large value of $k$. We now take $R_{k}=l_{k}+L_{k}+1$. This implies that $D_{k}$ holds for the set $S_{R_{k}}$ if and only if there is no open vertical crossing of the set $[0,l_{k}]\times[0,L_{k}]$. To estimate the probability of this event we define
\begin{equation}
J(x)=\left\{\begin{array}{cl}
 \eta_{0}(x)=0 \text{ or there is a poisson clock in a} \\ \text{neighboring edge of } x \text{ that rings before time } 1\end{array}\right\}.
\end{equation}

\par There are two important observations about the events $(J(x))_{x \in \ZZ}$: First, observe that $\PP_{\rho}(J(x))=1-\rho e^{-1}$. The second fact is that if $|x-y| \geq 2$, then $J(x)$ and $J(y)$ are independent.

\par The choice of $R_{k}=l_{k}+L_{k}+1$ helps us to estimate
\begin{align*}
p_{k} & = \PP_{\rho_{k}}\left[\begin{array}{cl}
[0,l_{k}] \times [0,L_{k}] \text{ does not} \\ \text{have a vertical open crossing} \end{array}\right] \\
& = \PP_{\rho_{k}}\left[\begin{array}{cl}
\text{there exists } u \in [0,l_{k}) \text{ such that for all } x \in [0,l_{k}] \\ \text{ there exists } t \in [u,u+1) \text{ such that } \eta_{t}(x)=0 \end{array}\right]  \\
& \leq L_{k}\PP_{\rho_{k}}\left(\bigcap_{x \in [0,l_{k}]} J(x) \right) \leq L_{k}\PP_{\rho_{k}}\left(\bigcap_{x \in [0,l_{k}] \cap 2\NN} J(x) \right) \\
& \leq L_{k}(1-\rho_{k} e^{-1})^{\sfrac{l_{k}}{2}} \leq l_{k}^{\sfrac{3}{2}}(1-\rho_{0}e^{-1})^{\sfrac{l_{k}}{2}}.
\end{align*}

Now, if we take $k$ large enough, we conclude that $l_{k}^{4}p_{k} \leq 1$ for $R_{k}=L_{k}+l_{k}+1$. For such a choice of $k$, and fixing $R_{k}$ from now on, we can apply Theorem \ref{teo:percolation} to conclude the proof of Theorem \ref{teo:detection}.

\appendix
~
\section{Concentration inequalities}
~

\par Here we recall some results about concentration of measures. We begin by the classical result known as Azuma's inequality.

\begin{teo}[Azuma's Inequality] Let $ \{X_{k}\}_{k=1}^{n} $ be a collection of independent random variables and $t>0$. Assume that there exist constants $ \{c_{k}\}_{k=1}^{n} $ satisfying $ \PP[|X_{k}|\leq c_{k}]=1$ for all $k$. Then
\begin{enumerate}
\item
\hfil
$
\begin{aligned}[t]
\PP\left[\sum_{k=1}^{n}X_{k}-\EE(X_{k}) \geq t\right] \leq \exp\left\{\frac{-t}{2\sum_{k=1}^{n}c_{k}^{2}}\right\};
\end{aligned}
$
\item
\hfil
$
\begin{aligned}[t]
\PP\left[\sum_{k=1}^{n}X_{k}-\EE(X_{k}) \leq -t\right] \leq \exp\left\{\frac{-t}{2\sum_{k=1}^{n}c_{k}^{2}}\right\}.
\end{aligned}
$
\end{enumerate}

\end{teo}

\par We will not prove this theorem here, since it can be found as Theorem 6.2 of \cite{blm}. This theorem implies a concentration bound for binomial random variables, that we state now as a corollary:

\begin{cor}\label{binconcentration}
If $X$ is a random variable with distribution $\text{\emph{Binomial}}(n,p)$ and $t>0$:
\begin{enumerate}
\item
\hfil
$
\begin{aligned}[t]
\PP\left[X-np \geq t\right] \leq \exp\left\{\frac{-t}{2n}\right\};
\end{aligned}
$

\item
\hfil
$
\begin{aligned}[t]
\PP\left[X-np \leq -t\right] \leq \exp\left\{\frac{-t}{2n}\right\}.
\end{aligned}
$
\end{enumerate}
\end{cor}

\par Our next objective is to prove concentration bounds for Poisson random variables. This is done in the next lemmas:

%%%
\newconstant{c:poisson_concentration_1}
\newconstant{c:poisson_concentration_2}
%%%

\begin{lemma}\label{concentration_poisson}
Let $\lambda>0$, $t>0$ and $X \sim \text{\emph{Poisson}}(\lambda)$. There exist constants $\useconstant{c:poisson_concentration_1}>0$ and $\useconstant{c:poisson_concentration_2}>0$ such that
\begin{displaymath}
\PP[X\geq 2\lambda+t] \leq e^{-\useconstant{c:poisson_concentration_1} t},
\end{displaymath}
and
\begin{displaymath}
\PP[X \leq \sfrac{\lambda}{3}] \leq e^{-\useconstant{c:poisson_concentration_2} \lambda}.
\end{displaymath}
Moreover, the constants $\useconstant{c:poisson_concentration_1}>0$ and $\useconstant{c:poisson_concentration_2}>0$ do not depend on $t$ and $\lambda$.
\end{lemma}

\begin{proof}
Let $\useconstant{c:poisson_concentration_1} >0$ such that $ e^{\useconstant{c:poisson_concentration_1}}-1=2\useconstant{c:poisson_concentration_1} $. Then, by Markov's inequality,
\begin{align*}
\PP[X\geq 2\lambda+t] & =\PP[e^{\useconstant{c:poisson_concentration_1} X} \geq e^{\useconstant{c:poisson_concentration_1}(2\lambda+t)}] \\ & \leq \exp\{\lambda(e^{\useconstant{c:poisson_concentration_1}}-1)-\useconstant{c:poisson_concentration_1}(2\lambda+t)\}=e^{-\useconstant{c:poisson_concentration_1} t}.
\end{align*}

For the second inequality, observe that if $\theta>0$ then
\begin{align*}
\PP[X \leq \sfrac{\lambda}{3}] & =\PP[e^{-\theta X} \geq e^{\sfrac{-\theta \lambda}{3}}] \\ & \leq \exp\{\lambda(e^{-\theta}-1)+\sfrac{\theta \lambda}{3}\}=e^{-\lambda(-e^{-\theta}+1-\frac{\theta}{3})}.
\end{align*}
If we take $\theta$ small enough, we have $\useconstant{c:poisson_concentration_2} = -e^{-\theta}+1-\frac{\theta}{3}>0$.
\end{proof}

\begin{lemma}\label{concentration_poisson_2}
For all $\lambda>0$ and $X \sim \text{\emph{Poisson}} (\lambda)$
\begin{displaymath}
\PP[X \geq 3\lambda] \leq e^{-\lambda}.
\end{displaymath}
\end{lemma}
\begin{proof}
Take $\theta>0$ such that $3\theta=e^{\theta}$ and use the same computations as in the lemma above.
\end{proof}

%%%
\newconstant{c:poisson_concentration_3}
%%%

\begin{lemma}\label{concentration_poisson_3}
For all $\lambda>0$, $\mu \geq 3\lambda$ and $X \sim \text{\emph{Poisson}} (\lambda)$
\begin{displaymath}
\PP[X \geq \mu] \leq e^{-\useconstant{c:poisson_concentration_3}(\lambda+\mu)},
\end{displaymath}
for some positive constant $\useconstant{c:poisson_concentration_3}$.
\end{lemma}
\begin{proof}
Simply observe that
\begin{align*}
\PP[X\geq \mu] & =\PP[e^{X} \geq e^{\mu}] \\ & \leq \exp\{\lambda(e-1)-\mu\}\leq e^{-\lambda}e^{\mu(\frac{e}{3}-1)} \leq e^{-\useconstant{c:poisson_concentration_3}(\lambda+\mu)}.
\end{align*}
\end{proof}

\section{Heat kernel estimates}\label{ap_heatkernel}
~
\par In this section we prove heat kernel estimates for the symmetric random walk on $\ZZ$.

\par The heat kernel is defined as
\begin{displaymath}
p_{t}(x,y)=\PP_{x}[W_{t}=y],
\end{displaymath}
where $(W_{t})_{t \geq 0}$ is a continuous time simple symmetric random walk. It will be useful to consider also the discrete heat kernel, that is defined as
\begin{displaymath}
p_{n}(x,y)=\PP_{x}[X_{n}=y],
\end{displaymath}
where $(X_{n})_{n \in \NN}$ is a discrete time lazy symmetric random walk.

\par The next lemma gives us estimates in the discrete time case:

%%%
\newconstant{c:discrete_heat_kernel}
%%%

\begin{lemma}\label{lem:discrete_heat_kernel}
There exists a constant $\useconstant{c:discrete_heat_kernel}>0$ such that for all $n \in \NN$ and $x \in \ZZ$
\begin{displaymath}
p_{n}(0,x)\leq \frac{\useconstant{c:discrete_heat_kernel}}{\sqrt{n}}.
\end{displaymath}
\end{lemma}

\begin{proof}
If we write $\{Z_{n}\}_{n \in \NN}$ to the discrete time simple symmetric random walk, it is easy to see that $\left\{\frac{Z_{2n}}{2}\right\}_{n \in \NN}$ is a discrete time lazy symmetric random walk. This implies that the lemma is a consequence of
\begin{displaymath}
\PP_{0}[Z_{2n}=2x]\leq \frac{\useconstant{c:discrete_heat_kernel}}{\sqrt{n}}.
\end{displaymath}

Now we just need to count the number of paths that are in $2x$ at time $2n$. Assume $0 \leq x \leq n$ (by symmetry this extends to $-n \leq x \leq 0$, and this quantity is zero if $|x|>n$), and observe that the number of possible paths of the random walk that start at zero and is in $2x$ at time $2n$ is
$\left(\begin{array}{c} 2n \\ n+x\end{array}\right)$. With the aid of Stirling's approximation we estimate
\begin{displaymath}
p_{n}(0,x) =\PP_{0}[Z_{2n}=2x]=\frac{1}{2^{n}}\left(\begin{array}{c} 2n \\ n+x\end{array}\right) \leq
\frac{1}{2^{n}}\left(\begin{array}{c} 2n \\ n\end{array}\right) =\frac{(2n)!}{2^{n}(n!)^{2}} \leq \frac{\useconstant{c:discrete_heat_kernel}}{\sqrt{n}}.
\end{displaymath}
\end{proof}

\par Now we get analogous bounds for continuous time random walks:

%%%
\newconstant{c:continuous_heat_kernel}
%%%

\begin{prop}
For the continuous time random walk, there exists a constant $\useconstant{c:continuous_heat_kernel}>0$ such that for every $t \geq 0$ and $x \in \ZZ$
\begin{displaymath}
p_{t}(0,x) \leq \frac{\useconstant{c:continuous_heat_kernel}}{\sqrt{t}}.
\end{displaymath}
\end{prop}

\begin{proof}
We use a construction of the continuous time random walk with a Poisson process of rate 2 and a skeleton chain given by a lazy symmetric random walk. Let $N_{t}$ be the number of jumps in the interval $[0,t]$, that has distribution $\text{Poisson}(2t)$. We use Lemmas \ref{concentration_poisson} and \ref{lem:discrete_heat_kernel} to get the estimates:
\begin{align*}
p_{t}(0,x) & \leq \PP\left[N_{t} \leq \sfrac{2t}{3}\right] + \sum_{k=\lceil \sfrac{2t}{3} \rceil}^{+\infty} \PP[X_{k}=x\, ,\, N_{t}=k] \\
& \leq e^{-2\useconstant{c:poisson_concentration_2}t} + \sum_{k=\lceil \sfrac{2t}{3} \rceil}^{+\infty} \frac{c}{\sqrt{k}}\PP[N_{t}=k]  \leq e^{-2\useconstant{c:poisson_concentration_2}t} + \sum_{k=\lceil \sfrac{2t}{3} \rceil}^{+\infty} \frac{c}{\sqrt{t}}\PP[N_{t}=k] \\
& \leq e^{-2\useconstant{c:poisson_concentration_2}t} + \frac{c}{\sqrt{t}} \leq \frac{\useconstant{c:continuous_heat_kernel}}{\sqrt{t}}.
\end{align*}
\end{proof}

\bibliographystyle{plain}
\bibliography{mybib}

\begin{thebibliography}{10}

\bibitem{att}
Daniel Ahlberg, Vincent Tassion, and Augusto Teixeira.
\newblock Sharpness of the phase transition for continuum percolation in
  $\mathbb{R}^2$.
\newblock {\em arXiv preprint arXiv:1605.05926}, 2016.

\bibitem{avena}
Luca Avena, Renato~dos Santos, and Florian V{\"o}llering.
\newblock Transient random walk in symmetric exclusion: limit theorems and an
  einstein relation.
\newblock {\em arXiv preprint arXiv:1102.1075}, 2011.

\bibitem{br}
Bela Bollobas and Oliver Riordan.
\newblock {\em Percolation}.
\newblock Cambridge University Press, 2006.

\bibitem{blm}
St{\'e}phane Boucheron, G{\'a}bor Lugosi, and Pascal Massart.
\newblock {\em Concentration inequalities: A nonasymptotic theory of
  independence}.
\newblock OUP Oxford, 2013.

\bibitem{grimmett}
Geoffrey~R Grimmett.
\newblock Percolation (grundlehren der mathematischen wissenschaften).
\newblock 2010.

\bibitem{hhsst}
Marcelo Hil{\'a}rio, Frank~den Hollander, Vladas Sidoravicius, Renato
  Soares~dos Santos, and Augusto Teixeira.
\newblock Random walk on random walks.
\newblock {\em arXiv preprint arXiv:1401.4498}, 2014.

\bibitem{ff}
Fran{\c{c}}ois Huveneers, Fran{\c{c}}ois Simenhaus, et~al.
\newblock Random walk driven by simple exclusion process.
\newblock {\em Electronic Journal of Probability}, 20, 2015.

\bibitem{kkp}
George Kesidis, Takis Konstantopoulos, and Shashi Phoha.
\newblock Surveillance coverage of sensor networks under a random mobility
  strategy.
\newblock In {\em Sensors, 2003. Proceedings of IEEE}, volume~2, pages
  961--965. IEEE, 2003.

\bibitem{psss}
Yuval Peres, Alistair Sinclair, Perla Sousi, and Alexandre Stauffer.
\newblock Mobile geometric graphs: detection, coverage and percolation.
\newblock In {\em Proceedings of the twenty-second annual ACM-SIAM symposium on
  Discrete Algorithms}, pages 412--428. SIAM, 2011.

\bibitem{pt}
Serguei Popov and Augusto Teixeira.
\newblock Soft local times and decoupling of random interlacements.
\newblock {\em arXiv preprint arXiv:1212.1605}, 2012.

\bibitem{ss}
Vladas Sidoravicius and Alexandre Stauffer.
\newblock Phase transition for finite-speed detection among moving particles.
\newblock {\em Stochastic Processes and their Applications}, 125(1):362--370,
  2015.

\bibitem{ssz}
Vladas Sidoravicius and Alain-Sol Sznitman.
\newblock Percolation for the vacant set of random interlacements.
\newblock {\em Communications on Pure and Applied Mathematics}, 62(6):831--858,
  2009.

\bibitem{st}
Vladas Sidoravicius and Augusto Teixeira.
\newblock Absorbing-state transition for stochastic sandpiles and activated
  random walks.
\newblock {\em arXiv preprint arXiv:1412.7098}, 2014.

\bibitem{stauffer}
Alexandre Stauffer.
\newblock Space-time percolation and detection by mobile nodes.
\newblock {\em The Annals of Applied Probability}, 25(5):2416--2461, 2015.

\bibitem{s}
Alain-Sol Sznitman.
\newblock Vacant set of random interlacements and percolation.
\newblock {\em Annals of mathematics}, pages 2039--2087, 2010.

\end{thebibliography}

\end{document}